\documentclass[12pt]{article}
\usepackage[english]{babel}
\usepackage[dvips]{graphicx}
\usepackage{graphics}
\usepackage{graphicx}
\usepackage{rotating}
\usepackage{psfrag}
\usepackage{color}
\usepackage[dvips]{epsfig}
\usepackage[T1]{fontenc}
\usepackage{subcaption}
\usepackage{refstyle}
\usepackage{pgf}
\usepackage{amsfonts}
\usepackage{fancyhdr}
\usepackage[T1]{fontenc}
\usepackage{amsmath, amssymb, bm}
\usepackage{latexsym}
\usepackage{tabularx}
\usepackage{algorithm}
\usepackage{algorithmic}
 \def \rot{\mathrm{rot}}
 \def \div{\mathrm{div}}

\usepackage{amsmath,amsxtra,amssymb,latexsym, amscd,amsthm}
\usepackage{amsmath,amsfonts,amssymb,amsthm,amsxtra}
\usepackage{textcomp}
\usepackage{latexsym}
\usepackage{tocbibind}
\usepackage{authblk}
\def\eps{\varepsilon}

\newtheorem{prop}{Proposition}[section]

\newtheorem{lem}{Lemma}[section]
\newtheorem{definition}{Definition}[section]

\newtheorem{theorem}{Theorem}[section]

\newtheorem{problem}{Problem}[section]
\newcommand{\be}{\begin{equation}}
\newcommand{\ee}{\end{equation}}
\newcommand{\ba}{\begin{array}}
\newcommand{\ea}{\end{array}}

\newcommand{\bea}{\begin{eqnarray*}}
\newcommand{\eea}{\end{eqnarray*}}
\newcommand{\bean}{\begin{eqnarray}}
\newcommand{\eean}{\end{eqnarray}}

\newcommand{\R}{\mbox{\bf R}}

\def\Box{\leavevmode\vbox{\hrule
     \hbox{\vrule\kern5pt\vbox{\kern5pt}%
           \vrule}\hrule}}
         \usepackage{hyperref}  
\usepackage[toc,page]{appendix}
\usepackage{color}
\usepackage{cleveref}
\date{}

\usepackage{authblk}
\usepackage{tcolorbox}
\tcbuselibrary{listings,skins}
\usepackage{authblk}
\author[,1]{\small{Soumaya Oueslati}\thanks{soumaya.oueslati1@cyu.fr}}
\author[,1]{\small{Christian Daveau}  \thanks{christian.daveau@cyu.fr}}
\author[,1]{\small{Abil Aubakirov}    \thanks{abil.aubakirov@cyu.fr}}
\affil[1]{\textsc{\small{Department of Mathematics, CNRS(UMR 8088), University CY Cergy Paris,2 avenue Adolphe Chauvin, 95302 Cergy-Pontoise, France}}}

\begin{document}
\title{ A new variational formulation with high order impedance boundary condition for the scattering problem in electromagnetism}
\maketitle
\abstract{
In this paper, we propose some variational formulations with the use of high order impedance boundary condition (HOIBC) to solve the scattering problem.  We study the existence and uniqueness  of the solution. Then,  a discretization  of these formulations is done.
We give validations of the HOIBC obtained with a MoM code that show the improvement
in accuracy over the standard impedance boundary condition (SIBC) computations.}

\noindent {\bf Keywords}: boundary element method, scattering problem, Lagrange multipliers, high order impedance boundary condition.

\section{Introduction}\label{intro}
The paper deals with computational electromagnetic to model thin coatings on perfect conducting objects. The method of integral equations is often used
to calculate the far field radiated by a homogeneous object \cite{RFH}. It allows to set the equations on the surface of the scattering obstacle which lowers the dimension of the problem. In counterpart, the
equations are non local and then lead after discretization to dense linear systems. This method reduces drastically the number of unknowns compared to other finite elements methods where
thin coatings is meshed. \\
We sometimes use with integral methods the perfectly conducting assumption neglecting the losses. However, it is important to model the losses and finite conductivity
of the metal by an impedance boundary condition.  This impedance boundary condition defines a relationship between the tangential electric and magnetic fields on a surface.
Usually, the impedance of the coating is assumed to be
the same for all incidence angles and polarizations. It is called standard impedance boundary condition  or  Leontovicht condition. We can see the papers \cite{ABENDL, Stratton-Chu} about it, but the list is not exhaustive.\\
This assumption is valid for thin coatings with high refractive index or significant losses. However, it is known to be poorly efficient  for coatings with
smaller refractive index or lower losses \cite{DSW}.
In this framework, some approximations of the HOIBC are proposed in \cite{RYTOV_1940,KARP_1967,senior_1995} which involve at most a first derivative of the field. Other strategies have also been established in \cite{R-S,R-S_DEC, OMBS, Stupfel_NOV_2005,Stupfel_APR_2005} that based on a polynomial approximation of the impedance operator is done in the spectral domain.
 Then, the tangential traces of electric and magnetic fields are related with the incident angle at each point of the surface since the approximation of the impedance is defined
 as a ratio of polynomials of the sine of the incidence angle for each polarization. Numerical results for two-dimensional cylinders with lossy and lossless dielectric coatings and for
 three-dimensional bodies of revolution is given in these papers.
 Later, in \cite{Stupfel_JUNE_2011} sufficient uniqueness conditions (SUC) have been established for the solutions of Maxwell's equations associated withe the particular HOIBCs proposed in \cite{Stupfel_NOV_2005}.
 
 This paper deals with the implementation of high order impedance boundary condition implemented in an integral equation.
An outline of the  article is as follows: Section \ref{Math-phys} presents the physical model and the framework. In section \ref{ApproxIBC}, we give an approximation of the impedance operator with Hodge operator and we give approximation of the impedance boundary condition. To my knowledge, this is the first time that the impedance condition is written in this form.
In section \ref{section_SUC}, we determine SUC for the solutions of Maxwell's equations associated with the impedance conditions. 
In the following, we derive a formulation and in section \ref{Existence_uniqueness} we prove  the existence and uniqueness of a solution for this variational formulation. Sections \ref{Discretization} and \ref{Lagrange multipliers} is devoted to the  description of
the discretization. Numerical results are presented in section \ref{results} and in  section \ref{conclusion} concluding remarks are done.

\section{Mathematical model of physical problem}\label{Math-phys}
In this section, we recall the model problem on which we work. 
We consider the scattering problem of electromagnetic waves by a perfect conducting body with a complex coating. Let $\Omega$ be a bounded domain with a Lipschitz-continuous boundary $\Gamma$. Let $n$ be the unit normal vector to $\Gamma$ directed to the exterior of $\Omega$.
We define $\Omega_+$ as the space of radiating  fields $(\textbf{E},\textbf{H})$ solutions of Maxwell equations. 
The total electromagnetic fields $(\textbf{E}, \textbf{H})$ in $\Omega_+$ therefore have the expression:
\begin{equation*}
\left\{
 \begin{array}{llll}
   \mathbf{E} = \mathbf{E}^{inc}+\mathbf{E}^{sc} \\
   \mathbf{H} = \mathbf{H}^{inc}+\mathbf{H}^{sc} \\
\end{array}
\right.
\end{equation*}
wehre  $(\textbf{E}^{inc}, \textbf{H}^{inc})$ is incident fields and the fields $ (\textbf{E}^{sc}, \textbf{H}^{sc})$ scattered by the object $\Omega$.\\


%
%
\begin{figure}[H]\label{fig1}
\centering
  \includegraphics[scale=0.3]{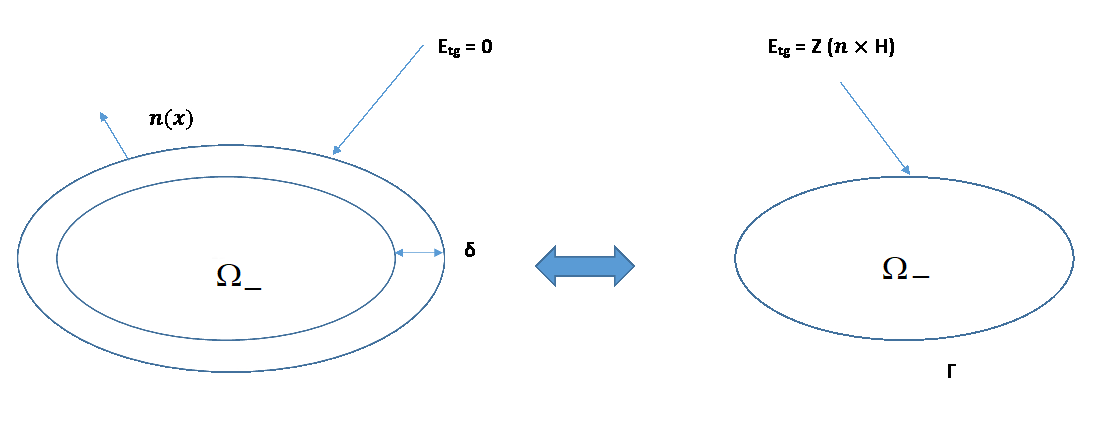}
  \caption{Equivalent impedance}
  \label{FigN1}
 \end{figure}
 
While the coating is replaced by the following impedance boundary condition on surface $\Gamma$, (see \Cref{FigN1}):
\begin{equation}\label{eq1.2}
\textbf{E}_{tg} - Z (n \times \textbf{H}) = 0 \,\,\,\,\,on\,\,\,\,\Gamma,
\end{equation}
where $Z$ is the impedance operator
, subscript tg denotes tangent component on the surface defined as:
$$\textbf{E}_{tg} = n \times(\textbf{E} \times n).$$
An approximation of  this impedance operator is given in the \Cref{ApproxIBC}, which depends  on layer thickness, on dielectric characteristics of medium of this layer. Also it depends on incident angle of incident electromagnetic plane wave. 


Waves propagate with constant wave number $k$ in the exterior unbounded domain $\Omega$. 
Clearly that the electromagnetic fields $(\textbf{E}, \textbf{H})$ is solution the following problem:
\begin{problem}\label{Max_Eq}
\begin{equation}
\left\{
 \begin{array}{llll}
 \nabla \times(\nabla\times\textbf{E} )-k^2 \textbf{E}=0\,\,\,\,\,\text{in} \,\,\, \Omega_+,  \\
 \nabla \times(\nabla\times\textbf{H} )-k^2 \textbf{H}=0\,\,\,\,\,\text{in} \,\,\, \Omega_+,  \\
   	   \textbf{E}_{tg} - Z (n \times \textbf{H}) = 0 \,\,\,\,\,on\,\,\,\,\Gamma.\\
  \lim_{r\rightarrow\infty} r(\mathbf{E}\times\mathbf{n}_r + \mathbf{H}) = 0.
\end{array}
\right.
\end{equation}
where $r = |\mathbf{x}|$ and $\mathbf{n}_r = \dfrac{\mathbf{x}}{|\mathbf{x}|}$.
\end{problem}   

Now, we shall show that problem \ref{Max_Eq} is well posed, using the result:
\begin{lem}\textbf{(Rellich)}: Let $\Omega$ be the open complement of a closed domain and $u\in L^2(\R^3 - \Omega)$ is a solution of the Helmholtz equation, satisfying
\begin{equation}
 \lim_{r\rightarrow\infty} \int_{|x| = r} |u(x)|^2 dx = 0 .
\end{equation}
Then $u = 0$, in $\R^3 - \Omega$.
\end{lem}

%
\begin{theorem}\label{SUC_theorem_1} 
The problem \ref{Max_Eq} admits a unique solution, if following relations are verified:
\begin{equation*}\label{SUC}
  \begin{cases}
    \Im(\mu) \leq 0, \\
    \Im(\epsilon) \leq 0, \\
    \Re (k_0 \int_{\Gamma}\mathbf{E}^* \cdot (\mathbf{n}\times\mathbf{H}) ds ) \geq 0. 
  \end{cases}
\end{equation*}
\end{theorem}
\begin{proof}
This theorem is an adaptation of uniqueness that is determined in \cite{Stupfel_JUNE_2011}. 
It gives us sufficient condition on characteristics of medium to get unique solution of the problem \ref{Max_Eq}.
That is how we get sufficient uniqueness condition (SUC) in general form
\begin{equation}
 \Re( k_0 \int_{\Gamma} (\mathbf{n}\times\mathbf{H}) \cdot \textbf{E}^*)\geq 0 
\end{equation}
\end{proof}
\section{Framework}
We present preliminary results concerning some basic functional spaces and operators which will be used throughout this work.
\subsection{Notation and spaces}
\begin{definition}
Let be $L_D$ and $L_R$ operators defined
for all vector function $\mathbf{A}$ sufficiently smooth, such that $\mathbf{A} \cdot \mathbf{n} = 0$ by:
\begin{equation}\label{LD_LR}
L_D (\mathbf{A}) =\nabla_{\Gamma}(\div_{\Gamma} \mathbf{A}), \ \ \ \ L_R(\mathbf{A})= \mathbf{\rot}_{\Gamma}(\rot_{\Gamma} \mathbf{A}).
\end{equation}
\end{definition}

Now, we consider some  the following Sobolev spaces on the surface $\Gamma$:

 \begin{enumerate}
\item $H^{-1/2}(\div, \Gamma)$  defined by    
$$H^{-1/2}(\div,\ \Gamma) = \{ \mathbf{g}\in (H^{-1/2}(\Gamma))^3,\ \mathbf{g}\cdot\mathbf{n} = 0,\ \div_{\Gamma}( \mathbf{g}) \in (H^{-1/2}(\Gamma))^3 \}. $$
It is an Sobolev spaces equipped with the norm:
$$\| \mathbf{g} \|_{-1/2,\div_{\Gamma}} = \left( \| \mathbf{g} \|^2_{-1/2,H(\Gamma)} + \| \div_{\Gamma}( \mathbf{g} )\|^2_{-1/2,H(\Gamma)} \right)^{1/2}$$
 \item And also the classical space $H^{-1/2}(\rot, \Gamma) $ defined on the boundary of $\Omega$
$$H^{-1/2}(\rot,\ \Gamma) = \{ \mathbf{g} \in H^{-1/2}(\Gamma)^3,\ \mathbf{g}\cdot\mathbf{n} = 0,\ \mathbf{rot}_{\Gamma}( \mathbf{g}) \in H^{-1/2}(\Gamma)^3 \} $$
equipped with the norm denoted by
$$\| \mathbf{g} \|_{-1/2,\rot_{\Gamma}} = \left( \| \mathbf{g} \|^2_{-1/2,H(\Gamma)} + \| \mathbf{rot}_{\Gamma} (\mathbf{g} ) \|^2_{-1/2,H(\Gamma)} \right)^{1/2}$$
  
\end{enumerate}

\begin{prop}
We have duality relations:
$$\left( H^{-1/2}(\rot,\ \Gamma) \right)' = H^{-1/2}(\div,\ \Gamma) \ \ and \ \ \left( H^{-1/2}(\div,\ \Gamma) \right)' = H^{-1/2}(\rot,\ \Gamma) $$
and $\forall \mathbf{f}\in H^{-1/2}(\div,\ \Gamma); \ \ \forall \mathbf{g}\in H^{-1/2}(\rot,\ \Gamma) $ we have
$$\| \mathbf{f} \|_{(-1/2,\rot_{\Gamma})'} = \sup_{\mathbf{g} \neq 0 \in H^{-1/2}(\rot,\Gamma) }{ \frac{<\mathbf{f},\mathbf{g}>}{\| \mathbf{g} \|_{-1/2,\rot_{\Gamma}}} } = \| \mathbf{f} \|_{-1/2,\div_{\Gamma}} $$
$$\| \mathbf{g} \|_{(-1/2,\div_{\Gamma})'} = \sup_{\mathbf{f} \neq 0 \in H^{-1/2}(\div,\Gamma) }{ \frac{<\mathbf{g},\mathbf{f} >}{\| \mathbf{f} \|_{-1/2,\div_{\Gamma} }} } = \| \mathbf{g} \|_{-1/2,\rot_{\Gamma}} $$
\end{prop}
%
%
%
%
\begin{theorem}{\label{2.7}}
We define a linear application $\gamma_0$ by:
$$\forall \mathbf{u} \in H^1 (\Omega),\ \ \gamma_0 \mathbf{u} = \mathbf{u}_{|_{\Gamma}} $$
Then $\gamma_0$ is continuous from $H^1 (\Omega)$ to $H^{1/2}(\Gamma)$ equipped with the norm and also we have:
$$\forall \mathbf{u}\in H^1 (\Omega),\,\, \|\gamma_0 \mathbf{u}\|_{1/2,H(\Gamma)} \leq C(\Gamma) \| \mathbf{u} \|_{H^1(\Omega)} $$
$ \forall \,\, \mathbf{f}\in H^{1/2}(\Gamma),\ \ \exists \mathbf{u} \in H^1(\Omega) $ such that 
$$\mathbf{f} = \gamma_0 \mathbf{u} \ \ and \ \ \| \mathbf{u} \|_{H^1(\Omega)} \leq C(\Gamma)\| \mathbf{f} \|_{1/2,H(\Gamma)}$$
\end{theorem}

\begin{proof}: See \cite{TE} pp.110-113. \end{proof}

%
%
%
%
%

\begin{theorem}
We define linear application $\gamma_n$ as:
$$\forall \mathbf{u}\in [H^1(\Omega)]^3, \ \ \gamma_n \mathbf{u} = - \mathbf{n} \times \mathbf{u}_{|_{\Gamma}} $$
Then, $\gamma_n$ extends in unique form to a continuous linear application from $H(\rot,\Omega)$ to $H^{-1/2}(\div, \Gamma)$ and also we have:
$$\forall \mathbf{u}\in H(\rot,\Omega), \ \ \|\gamma_n \mathbf{u}\|_{-1/2,\div_{\Gamma} } \leq C(\Gamma) \| \mathbf{u} \|_{H(\rot, \Omega) }$$
$\forall \mathbf{f} \in H^{-1/2}(\div,\Gamma),\ \exists \mathbf{u}\in H(\rot, \Omega)$   such\ \ that\ \ 
$$ \mathbf{f}= \gamma_n \mathbf{u}\ \ and \ \ \|\mathbf{u}\|_{H( \rot, \Omega)}\leq C(\Gamma)\|\mathbf{f}\|_{-1/2,\div_{\Gamma}}$$
\end{theorem}

\begin{proof}: See \cite{TE} pp.124-125 \end{proof}

\begin{theorem}
We define linear application $\Pi_{\Gamma}$ as:
$$\forall \mathbf{u}\in [H^1(\Omega)]^3, \ \ \Pi_{\Gamma} \mathbf{u} = - \mathbf{n} \times (\mathbf{n} \times \mathbf{u}_{|_{\Gamma}}) $$
Then, $\Pi_{\Gamma}$ extends in unique form to a continuous linear application from $H(\rot,\Omega)$ to $H^{-1/2}(\rot, \Gamma)$ and also we have:
$$\forall \mathbf{u}\in H(\rot,\Omega), \ \ \|\Pi_{\Gamma} \mathbf{u}\|_{-1/2,H^{}(\rot, \Gamma) } \leq C(\Gamma) \|\mathbf{u}\|_{H( \rot, \Omega)}$$
$\forall \mathbf{g}\in H^{-1/2}(\rot,\Gamma),\ \exists \mathbf{u}\in H(\rot, \Omega)$ such  that 
$$\mathbf{g} = \Pi_{\Gamma} \mathbf{u}\ \ and \ \ \|\mathbf{u}\|_{H( \rot, \Omega) }\leq C(\Gamma)\|\mathbf{g}\|_{-1/2,\rot_{\Gamma}}$$
\end{theorem}

\begin{proof}: See \cite{TE} pp.125-126 \end{proof}
\subsection{Properties of integral operators}
%
%
%
%
%
%
Some solutions of the problems of diffraction of electromagnetic waves can be expressed with the help of potential and in particular the potential of single and dual layer defined on the surface of the obstacle. We present here the properties of the harmonic potential.
\begin{definition}
We introduce the potentials $(B-S)$ and $Q$, that are defined by
\begin{equation}
(B-S)\mathbf{J} := \int_{\Gamma} \left( G(x,y)\mathbf{J}(y) +\frac{1}{k^2}\nabla_x G(x,y) \div_{\Gamma}\mathbf{J} \right) d\Gamma(y)
\end{equation}
\begin{equation}
Q\mathbf{M} := \int_{\Gamma} \nabla_y G(x,y) \times \mathbf{M} d\Gamma(y)
\end{equation}
and $G(x,y)$ is the Green kernel giving the outgoing solutions to the scalar Helmholtz equation.
\end{definition}
We have some results \cite{Lange-1995}, 
\begin{theorem}\label{thm_oper_Q}
The operator $Q$ is continuous from $H^{-1/2}(\div,\Gamma)$ to $H^{-1/2}(\rot, \Gamma)$ and we have that:
\begin{equation}
 |(\mathbf{n}\times Q + \frac{I}{2}) \mathbf{M} |_{-1/2,\div_{\Gamma}} \leq C | \mathbf{M} |_{-1/2,\div_{\Gamma}} \ \ \forall \mathbf{M}\in H^{-1/2}(\div,\Gamma)
\end{equation}
\end{theorem}

\begin{theorem}\label{thm_oper_BmS}
The operator $(B-S)$ is an isomorphisme from $H^{-1/2}(\div,\Gamma)$ to $H^{-1/2}(\rot,\Gamma)$ and it verifies the inequality:
\begin{equation}
 \| (B-S) \boldsymbol\phi \|_{-1/2,\rot_{\Gamma} } \leq C \| \boldsymbol\phi \|_{-1/2,\div_{\Gamma} }
\end{equation}
and the coercivity relation $\forall \boldsymbol\phi \in H^{-1/2}(\div, \Gamma)$:
\begin{equation}
 \Re (<\boldsymbol\phi, (B-S)\boldsymbol\phi>) \geq C \| \boldsymbol\phi \|^2_{-1/2,\div_{\Gamma} }
\end{equation}
\end{theorem}
 \section{Obtaining of the higher order impedance boundary}\label{ApproxIBC}
Generally, we take the constant impedance operator, known as standard or Leontovitch impedance boundary conditions \cite{R-S, Senior-Volakis, Lange-1995}.
Here, we are interested in approximation of the impedance operator which depends on incident angle.
We consider a plane isotropic medium with a local orthogonal basis (x,y,z) on a tangent plane where a normal vector $\textbf{n}$ is in z-direction and (x,y)
is a tangent plane \cite{R-S,OMBS}. The exact impedance tensor is defined  obtained for the dielectric plane layer \cite{R-S} by:
\begin{equation}\label{Z_ex_mix}
 Z_{xy}(k_x, k_y) = Z_{yx}(k_x, k_y) = i \sqrt{\frac{\mu}{\epsilon}} \dfrac{k_x k_y}{k k_z} \tan[k_z d],
\end{equation}
\begin{equation}\label{Z_ex_TETM}
 Z_{xx}(k_x, k_y) = -i \sqrt{\frac{\mu}{\epsilon}} \dfrac{k^2_x k^2_z + k^2_y k^2}{k k_z (k^2_x + k^2_y)} \tan[k_z d],
\end{equation}
and
\begin{equation}
 Z_{yy}(k_x, k_y) = -i \sqrt{\frac{\mu}{\epsilon}} \dfrac{k^2_y k^2_z + k^2_x k^2}{k k_z (k^2_y + k^2_x)} \tan[k_z d],
\end{equation}
$Z$ is a impedance tensor of wave numbers ($k_x$, $k_y$), wave frequency and coating at each point of the surface.
 In \cite{R-S} a method based on a spectral domain approach is presented to derive high order impedance boundary condition. The spectral domain approach is applicable
to complex coatings and surface treatments as well as simple dielectric. So, the impedance boundary condition is approximated as a ratio of second order
polynomials for a coating invariant under rotation. Then they transform the polynomial approximation of the spectral domain boundary condition into the spatial domain using
elementary properties of the Fourier transform. Then they get
\begin{equation}\label{hoibc_r_s_x}
 (1 + b_1 \partial^2_x + b_2\partial^2_y)E_x + (b_1 - b_2)\partial^2_{xy}E_y = (a_1 - a_2)\partial^2_{xy}H_x - (a_0 + a_1\partial^2_x +
 a_2\partial^2_y)H_y,
\end{equation}
and
\begin{equation}\label{hoibc_r_s_y}
 (b_1 -b_2)\partial^2_{xy}E_x + (1 + b_2\partial^2_x + b_1\partial^2_y)E_y = (a_0 + a_2\partial^2_x + a_1\partial^2_y)H_x
 + (a_2 - a_1)\partial^2_{xy}H_y.
\end{equation}
where $E_{tg}=(E_x,E_y)$ and $H_{tg}=(H_x,H_y)$.
\subsection{ The impedance boundary condition form in 3D case }
%
%
We use formula (\ref{LD_LR}) to rewrite the equations (\ref{hoibc_r_s_x}) and (\ref{hoibc_r_s_y}) with operators $L_D$ and $L_R$ as following equations:
$$
 \mathbf{x}\cdot(\mathbf{E}_{tg} + b_1 L_D\mathbf{E}_{tg} - b_2 L_R\mathbf{E}_{tg}) = \mathbf{x}\cdot(a_0(\mathbf{n}\times\mathbf{H})
 + a_1 L_D(\mathbf{n}\times\mathbf{H}) - a_2 L_R(\mathbf{n}\times\mathbf{H})),
$$
and
$$
 \mathbf{y}\cdot(\mathbf{E}_{tg} + b_1 L_D\mathbf{E}_{tg} - b_2 L_R\mathbf{E}_{tg}) = \mathbf{y}\cdot(a_0(\mathbf{n}\times\mathbf{H})
 + a_1 L_D(\mathbf{n}\times\mathbf{H}) - a_2 L_R(\mathbf{n}\times\mathbf{H})).
$$
%
Then, we obtain the following approximation of the high order impedance boundary condition:
\begin{equation}\label{HOIBC3D}
 Z_{3D}: \ \  (I + b_1 L_D - b_2 L_R)\textbf{E}_{tg} = (a_0 I + a_1 L_D - a_2 L_R)(\mathbf{n} \times \textbf{H}).
\end{equation}

\subsection{ The impedance boundary condition form in 2D case }
%
Here we need to consider two different situations. In the first one, we consider case when the electric field is perpendicular to the incident plane, as shown on (\Cref{Polarizations} a). Incident, scattered and transverse electric fields are directed toward the viewer. The direction of magnetic field was chosen such that energy current has positive direction, i.e. direction of wave propagation. We call this case, transverse-electric (TE) polarization. 
In the second case electric fields are parallel to incident plane, as shown in (\Cref{Polarizations} b). In this case we call it transverse-magnetic (TM) polarization.
\begin{figure}[H]
\begin{minipage}{0.45\textwidth}
\begin{figure}[H]
\centering
        \includegraphics[scale=0.5]{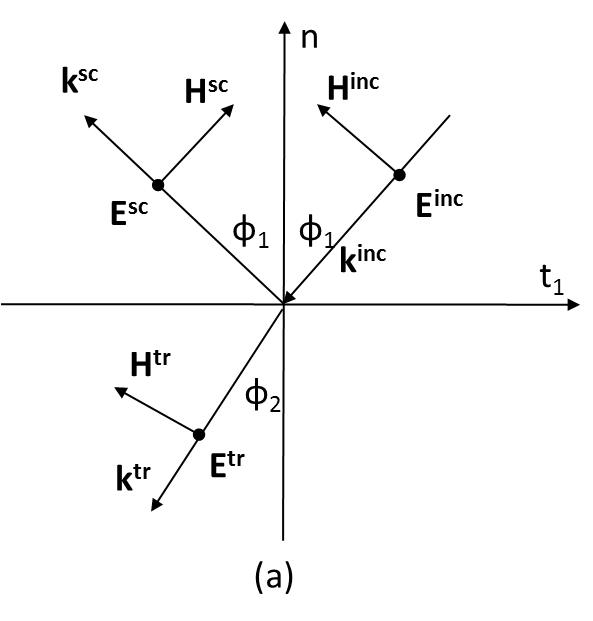}
    \end{figure}
\end{minipage}   
\begin{minipage}{0.45\textwidth}
\begin{figure}[H]
\centering
\includegraphics[scale=0.5]{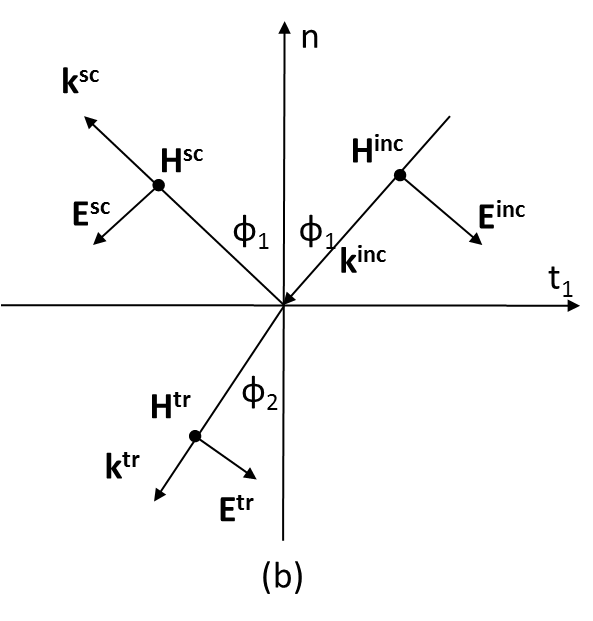}
\end{figure}
\end{minipage}
\caption{Reflection and refraction with (a) TE and (b) TM polarizations.}
 \label{Polarizations}
\end{figure}
%
%
%
We assume that the incident fields propagate perpendicular to the cylinder axis, so that $\partial / \partial y = 0$. 
These relations (\ref{hoibc_r_s_x}) and (\ref{hoibc_r_s_y}) are written in matrix form:
\begin{equation}\label{hoibc_2dmat}
 \left( \begin{array}{cc}
         1 + b_1 \partial^2_x & 0 \\
         0 & 1 + b_2\partial^2_x
        \end{array} \right)
 \left( \begin{array}{c}
         E_x \\ E_y
        \end{array} \right) =
 \left( \begin{array}{cc}
         a_0 + a_1 \partial^2_x & 0 \\
         0 & a_0 + a_2\partial^2_x
        \end{array} \right)
 \left( \begin{array}{c}
         -H_y \\ H_x
        \end{array} \right)
\end{equation}
%
%
%
%
%
%
%
%
%
We get first order approximation of impedance in two dimensional cases for each polarization
\begin{equation}\label{Z_2Dj_boxed}
Z_{2Dj}: \ \ (1 + b_j \partial^2_x)\mathbf{E}_{tg} = (a_0 + a_j \partial^2_x)\mathbf{n}\times\mathbf{H}
\end{equation}
where $j = 1,2$ correspond to TE and TM polarizations, respectively.
%
In the same manner,  the equation (\ref{Z_2Dj_boxed}) can be extended to second order polynomials in $\partial^2_x$. Also, we obtain the following approximation of the high order impedance boundary  in two dimensional:
\begin{equation}\label{i_eq}
\mathbf{E}_{tg} + b_j \partial^2_x \mathbf{E}_{tg} + b'_j \partial^4_x \mathbf{E}_{tg} = a_0 (\mathbf{n}\times\mathbf{H}) + a_j \partial^2_x (\mathbf{n}\times\mathbf{H}) + a'_j \partial^4_x (\mathbf{n}\times\mathbf{H})
\end{equation}
or it can be reduced to constant:
\begin{equation}\label{ii_eq}
\mathbf{E}_{tg} = a_0 (\mathbf{n}\times\mathbf{H}) .
\end{equation}
We will call the equation (\ref{Z_2Dj_boxed}) as first order IBC (IBC1). And we will call the equation (\ref{i_eq}) second order IBC (IBC2), the equation (\ref{ii_eq}) zeroth order IBC (IBC0), which is also known as Leontovich IBC.
%
\section{Study of Maxwell's equations with HOIBC}\label{section_SUC}
 We want to establish SUC for the solutions of Maxwell's equations associated with these IBCs, using of the uniqueness theorem mentioned earlier.
 We begin this section by  recalling the Maxwell's equations 
\begin{problem}\label{Max_Eq_SUC0}
\begin{equation}
\left\{
 \begin{array}{llll}
 \nabla \times(\nabla\times\textbf{E} )-k^2 \textbf{E}=0\,\,\,\,\,\text{in} \,\,\, \Omega_+,  \\
 \nabla \times(\nabla\times\textbf{H} )-k^2 \textbf{H}=0\,\,\,\,\,\text{in} \,\,\, \Omega_+,  \\
  \lim_{r\rightarrow\infty} r(\mathbf{E}\times\mathbf{n}_r + \mathbf{H}) = 0.
\end{array}
\right.
\end{equation}
\end{problem}
On the first time, we are going to study the problem \ref{Max_Eq_SUC0} associated with a more complicated boundary condition, that is called first order IBC1:
%
%
%
\begin{equation}\label{CM1}
 \mathbf{E}_{tg} + b_j L(\mathbf{E}_{tg}) = a_0 (\mathbf{n} \times \mathbf{H}) + a_j L(\mathbf{n} \times \mathbf{H}) \ \ on \ \Gamma,
\end{equation}
where $L$ is a complex differential operator, also known as \textbf{Hodge} operator. We assume that
\begin{equation}\label{CM1_CSB}
 \begin{cases}
 (\mathbf{n} \times \mathbf{H})\cdot\tau = (\mathbf{n} \times \mathbf{H})\cdot\nu = 0\ \  \text{on} \ \partial{\Gamma}\ \ \text{if} \ \ \Gamma\ \ \text{is}\ \ \text{open}, \\
 \mathbf{E}_{tg}\cdot\tau = \mathbf{E}_{tg}\cdot\nu = 0\ \ \text{on} \ \partial{\Gamma}\ \ \text{if}\ \ \Gamma\ \ \text{is}\ \ \text{open}. \\
 \end{cases}
\end{equation}
We have the following result
\begin{theorem}
The Maxwell problem (\ref{Max_Eq_SUC0}) with the boundary conditions (\ref{CM1}) and (\ref{CM1_CSB}) has a unique solution if the following relations are verified
 \begin{equation}\label{SUC_1}
  \begin{cases}
   \Im(\mu) \leq 0, \\
   \Im(\epsilon) \leq 0, \\
   a_j - b_j^* a_0 \neq 0, \\
   \Re (a_j - b_j^* a_0) = 0, \\
   \Im(a_0^* a_j)\Im(a_j - b_j^* a_0) \geq 0, \\
   \Im(b_j)\Im(a_j - b_j^* a_0) \geq 0.
  \end{cases}
 \end{equation}
\end{theorem}
\begin{proof}
See \cite{Stupfel_JUNE_2011}.
\end{proof}

On the other hand, we study yet another more complicated HOIBC, which is called second-order IBC that increases power of Hodge operator.
%
%
%
For the SUC of the problem (\ref{Max_Eq_SUC0}) with IBC2, we assume that
\begin{equation}\label{CM2}
 \mathbf{E}_{tg} + b_j L(\mathbf{E}_{tg}) + b'_j L^2(\mathbf{E}_{tg}) = a_0 (\mathbf{n} \times \mathbf{H}) + a_j L(\mathbf{n} \times \mathbf{H}) + a'_j L^2 (\mathbf{n} \times \mathbf{H}) \ \ on \ \Gamma ,
\end{equation}
where $L^2(\cdot) = L \circ L (\cdot)$. And all coefficients $a_j, a'_j, b_j$ and $b'_j$ are defined locally and depend on incident angle. With a next conditions on a bound
\begin{equation}\label{CM2_CSB}
\begin{cases}
 (\mathbf{n} \times \mathbf{H})\cdot\tau = (\mathbf{n} \times \mathbf{H})\cdot\nu = 0\ \ on\ \partial{\Gamma}\ \ if\ \ \Gamma\ \ is\ \ open, \\
 \mathbf{E}_{tg}\cdot\tau = \mathbf{E}_{tg}\cdot\nu = 0\ \ on\ \partial{\Gamma}\ \ if\ \ \Gamma\ \ is\ \ open \\
 X_1\cdot\nu = X_2\cdot\nu = Y_1\cdot\nu = Y_2\cdot\nu = 0\ \ on\ \partial{\Gamma}\ \ if\ \ \Gamma\ \ is\ \ open \\
 X_1\cdot\tau = X_2\cdot\tau = Y_1\cdot\tau = Y_2\cdot\tau = 0\ \ on\ \partial{\Gamma}\ \ if\ \ \Gamma\ \ is\ \ open.
 \end{cases}
\end{equation}
where
\begin{equation}\label{CM2_XY}
 \begin{cases}
 X_1 = \nabla_{\Gamma} \div_{\Gamma} \mathbf{E}_{tg} = L_D \mathbf{E}_{tg} \\
 X_2 = \nabla_{\Gamma} \div_{\Gamma} (\mathbf{n} \times \mathbf{H}) = L_D (\mathbf{n} \times \mathbf{H}) \\
 Y_1 = \mathbf{rot}(\mathbf{n} (\rot_{\Gamma} \mathbf{E}_{tg}) ) = L_R \mathbf{E}_{tg}\\
 Y_2 = \mathbf{rot}(\mathbf{n} (\rot_{\Gamma} (\mathbf{n} \times \mathbf{H})) ) = L_R (\mathbf{n} \times \mathbf{H})
 \end{cases}
\end{equation}

So for the IBC2 we have following uniqueness theorem

\begin{theorem}
The Maxwell problem (\ref{Max_Eq_SUC0}) with a boundary condition (\ref{CM2}), (\ref{CM2_CSB}) and (\ref{CM2_XY}) admits a unique solution if next relations are verified
\begin{equation}
  \begin{cases}
   \Im(\mu) \leq 0, \\
   \Im(\epsilon) \leq 0, \\
   \Delta \neq 0, \\
   \Re[\Delta^* (a_j b'^*_j - a'_j b_j^* )] =0,\\
   \Re[\Delta^* (a'_j  - a_0 b'^*_j )] = 0,\\
   \alpha\Im{b'_j} + \beta\Im(b_j b'^*_j) \leq 0,\\
   \alpha \Im(a'_j a_0^*) - \beta \Im(a'_j a_j^*) \leq 0,\\
   -\alpha\Im{b_j} + \beta\Im{b'_j} \leq 0,\\
   \alpha\Im(a_j a_0^*) - \beta\Im(a'_j a_0^*) \geq 0
  \end{cases}
\end{equation}
with
$$ \Delta = (a_0 b_j^* - a_j)(a_j b'^*_j - a'_j b_j^*) - (a'_j - a_0 b'^*_j)^2 $$
$$\alpha = \Im[\Delta^* (a_j b'^*_j - a'_j b_j^*)] $$
$$\beta =\Im[\Delta^* (a_0 b'^*_j - a'_j)] $$
\end{theorem}

\begin{proof}
see \cite{Stupfel_JUNE_2011}.
\end{proof}
All these three theorems are based on the theorem \ref{SUC_theorem_1} and its condition $$\Re (k_0 \int_{\Gamma}\mathbf{E}^* \cdot (\mathbf{n} \times \mathbf{H}) ds) \geq 0.$$
 So, to verify this condition, coefficients of HOIBC should verify conditions mentioned in these theorems.
\section{Calculation of the coefficients}
\subsection{The coefficients of the IBC1}
Now, we propose some different methods to calculate these coefficients $(a_0, a_j, b_j)$. The impedance $Z_{2Dj}$ (\ref{Z_2Dj_boxed}) is the following rational function of $k^2_x$
\begin{equation}\label{1.55}
Z_{2Dj} = \frac{a_0 - a_j k^2_x}{1 - b_j k^2_x}, \ \ j=1,2
\end{equation}
The coefficients indicated by $j = 1,2$ correspond to polarizations TE and TM respectively.
From (\ref{Z_ex_mix}) and (\ref{Z_ex_TETM}) we can express exact impedance for TE and TM polarization:
\begin{equation}\label{Exact_Imp_TE}
 Z^{ex}_{TE}=\sqrt{\frac{\mu}{\epsilon}}\frac{k_z}{k}\tan{(k_z d)} = z_0 \sqrt{\mu_r \epsilon_r - \left(\frac{k_x}{k_0}\right)^2} \tan{\left(\sqrt{\mu_r \epsilon_r - \left(\frac{k_x}{k_0}\right)^2}k_0 d\right)} / \epsilon_r
\end{equation}
\begin{equation}\label{Exact_Imp_TM}
 Z^{ex}_{TM}=\sqrt{\frac{\mu}{\epsilon}}\frac{k}{k_z}\tan{(k_z d)} = \frac{z_0 \mu_r  \tan{\left(\sqrt{\mu_r \epsilon_r - \left(\frac{k_x}{k_0}\right)^2}k_0 d\right)} }{ \sqrt{\mu_r \epsilon_r - \left(\frac{k_x}{k_0}\right)^2} }
\end{equation}
\begin{itemize}
\item \textbf{First method}\\
The accuracy of the boundary condition may be improved by including the coefficient $a_1$, thus approximation derive first order polynomial form \cite{R-S}. Wave formulas of electromagnetic fields give the quantity of wave number $x$-component, at calculation coefficients of approximate polynomials. Where from we denote $\xi = -(\frac{k_x}{k_0})^2 = -\sin^2(\theta)$  and input it in the exact impedance values (\ref{Exact_Imp_TE}), (\ref{Exact_Imp_TM}), we also obtain

%
$$Z^{ex}_{TE}(\xi) = z_0 \sqrt{\mu_r \epsilon_r + \xi } \tan{\left(\sqrt{\mu_r \epsilon_r + \xi}k_0 d\right)} / \epsilon_r$$
$$Z^{ex}_{TM}(\xi) = \frac{z_0 \mu_r  \tan{\left(\sqrt{\mu_r \epsilon_r +\xi } k_0 d\right)} }{ \sqrt{\mu_r \epsilon_r + \xi } } .$$
Here we approximate impedance as ratio of polynomials of $\xi$.

If $\theta = 0$, we get that $a_0 = Z(0) = Z^{ex}(0)$, for a normally incident wave, which is known as the Leontovich boundary condition and we get
$$a_0 = \sqrt{\frac{\mu_0\mu_r}{\epsilon_0\epsilon_r}} \tan{\left( \omega\sqrt{\mu_0\mu_r\epsilon_0\epsilon_r} d \right)} $$
We calculate other coefficients $a_j$ and $b_j$, using two arbitrary angles $\theta_1$ and $\theta_2$ by:
$$ \left( 
\begin{array}{cl} a_j \\ b_j
\end{array} \right) = \left[ 
\begin{array}{clcr}
-k^2_x(\theta_1) &  k^2_x Z_j^{ex}(\theta_1)\\
-k^2_x(\theta_2) &  k^2_x Z_j^{ex}(\theta_2)
\end{array} \right]^{-1}
\left( 
\begin{array}{cl} Z_j^{ex}(\theta_1) - a_0 \\ Z_j^{ex}(\theta_2) - a_0
\end{array} \right)
$$
The indices correspond to TE and TM polarizations, as in (\ref{1.55}).
%
%
%
%
%
Therefore, $a_j$ is defined for arbitrary $\theta_1 \in ]0, \pi/2 [$
$$a_j = \frac{Z^{ex}_j(\theta_1) - a_0}{\xi(\theta_1)} \ \ \text{and} \ \ b_j = 0 .$$

We can approximate impedance $Z^{ex}$ as the first order Taylor polynomial for $\xi$ near to zero, which yields
$$a_1 = z_0\frac{k_0 d}{2 \epsilon_r } + z_0\frac{ \tan (\sqrt{\mu_r\epsilon_r} k_0 d)}{2 \sqrt{\mu_r\epsilon_r} \epsilon_r} + z_0\frac{k_0 d \tan^2(\sqrt{\mu_r\epsilon_r} k_0 d)}{2 \epsilon_r} \ \ \ (Taylor)$$
$$b_1 = 0 .$$

Now we suppose that $b_1$ is different to zero and we approximate the impedance as ratio of polynomials. Stupfel mentioned Pad\'e approximation, which considers second order Taylor approximation as follows
\begin{equation}\label{Pade}
Z^{ex}_1(\xi) = c_0 + c_1\xi + c_2\xi^2 + O(\xi^2) \approx \frac{a_0 + a_1\xi}{1 + b_1\xi} ,
\end{equation}
where coefficients $c_0, c_1$ and $c_2$ are Taylor coefficients:
$$c_0 = z_0\sqrt{\frac{\mu_r}{\epsilon_r}} \tan{\left( \sqrt{\mu_r\epsilon_r} k_0 d \right)} $$
$$c_1 = z_0\frac{k_0 d}{2 \epsilon_r } + z_0\frac{ \tan (\sqrt{\mu_r\epsilon_r} k_0 d)}{2 \sqrt{\mu_r\epsilon_r} \epsilon_r} + z_0\frac{k_0 d \tan^2(\sqrt{\mu_r\epsilon_r} k_0 d)}{2 \epsilon_r} $$
$$c_2 = \frac{z_0 k_0 d}{8 \epsilon^2_r \mu_r} + \left( \frac{z_0 k^2_0 d^2}{4 \epsilon_r \sqrt{\epsilon_r \mu_r} } - \frac{z_0}{8 \epsilon_r (\epsilon_r \mu_r)^{3/2}} \right)\tan(\sqrt{\epsilon_r \mu_r} k_0 d) $$
$$+ \frac{z_0 k_0 d}{8 \epsilon^2_r \mu_r} \tan^2(\sqrt{\epsilon_r \mu_r} k_0 d) + \frac{z_0 k^2_0 d^2}{2 \epsilon_r \sqrt{\epsilon_r \mu_r}} \tan^3(\sqrt{\epsilon_r \mu_r} k_0 d) .$$
By multiplying this Taylor polynomial by denominator polynomial and equating the product to the numerator polynomial, we derive equations, where coefficients of Pade approximation may be determined through coefficients of Taylor approximation. So, from (\ref{Pade}) we have
$$a_0 + a_1\xi \approx c_0 + (c_1 + c_0 b_1)\xi + (c_2 + c_1 b_1)\xi^2 + O(\xi^2) ,$$
%
%
%
%
%
Finally, we get the coefficients of Padé approximation
$$b_1 = - \frac{c_2}{c_1}, \,\,\,\,\,\, a_1 = c_1 - c_0 \frac{c_2}{c_1} .$$
\item  \textbf{Second method}\\
We present another method to calculate the coefficients. For different $\theta_1$, $\theta_2$ from $] 0, \pi/2[$, we get different values of $\xi_i = - k^2_0 \sin^2(\theta_i)$ and two linear equations \cite{R-S}
$$Z_1^{ex}(\xi_1) - b_1 \xi_1 Z_1^{ex}(\xi_1) - a_0 + a_1 \xi_1 = 0 $$
$$Z_1^{ex}(\xi_2) - b_1 \xi_2 Z_1^{ex}(\xi_2) - a_0 + a_1 \xi_2 = 0 .$$
That can be solved in matrix form as follows
$$ \left( 
\begin{array}{cl} a_j \\ b_j
\end{array} \right) = \left[ 
\begin{array}{clcr}
\xi_1 &  -\xi_1 Z_j^{ex}(\xi_1)\\
\xi_2 &  -\xi_2 Z_j^{ex}(\xi_2)
\end{array} \right]^{-1}
\left( 
\begin{array}{cl} Z_j^{ex}(\xi_1) - a_0 \\ Z_j^{ex}(\xi_2) - a_0
\end{array} \right)
$$
and we get the coefficients of so-called Collocation approximation \cite{Stupfel_MARCH_2000}
$$a_1 = \frac{-\xi_2 Z_j^{ex}(\xi_2) (Z_j^{ex}(\xi_2) - a_0) + \xi_1 Z_j^{ex}(\xi_1)(Z_j^{ex}(\xi_2) - a_0 )}{\xi_1\xi_2(Z_j^{ex}(\xi_1) - Z_j^{ex}(\xi_2))} $$
$$b_1 = \frac{ -\xi_2 (Z_j^{ex}(\xi_1) - a_0) + \xi_1 (Z_j^{ex}(\xi_2) - a_0 ) }{\xi_1\xi_2(Z_j^{ex}(\xi_1) - Z_j^{ex}(\xi_2))} .$$
\end{itemize}

Note that in two dimensional case we suppose that $\partial_y = 0$. The three dimensional impedance $Z_{3D}$ should correspond to $Z_{2Dj}$ in two dimensional case. Coefficients $a_1, b_1$ and $a_2, b_2$ should correspond to those in two dimension, TE and TM coefficients, respectively.
\subsection{The coefficients of the IBC2}
We can apply all methods that were proposed earlier for IBC2, too. We get Taylor approximation, if we admit that $b_j$ and $b'_j$ are zero. If they are not zero, we get ratio of polynomials approximation. Pad\'e approximation uses fourth order Taylor expansion. Whence we get five equation for five unknown coefficients. And in the last method, the coefficients are calculated by solving system of linear equations for different $\theta_k \in ]0, \pi/2[,\ k = 1,2,3,4$
$$\left( 
\begin{array}{cccc} a_j \\ a'_j \\ b_j \\ b'_j
\end{array} \right) = \left[ 
\begin{array}{cccc}
\xi_1 & \xi^2_1 & -\xi_1 Z_j^{ex}(\xi_1) & -\xi^2_1 Z_j^{ex}(\xi_1)^2 \\
 ... & ... & ... & ... \\
\xi_4 & \xi^2_4 & -\xi_4 Z_j^{ex}(\xi_4) & -\xi^2_4 Z_j^{ex}(\xi_4)^2 \\
\end{array} \right]^{-1}
\left( 
\begin{array}{cl} Z_j^{ex}(\xi_1) - a_0\\ ... \\ ... \\ Z_j^{ex}(\xi_4) - a_0
\end{array} \right)
$$
Here, we suppose that in the last two methods we get invertible matrices.
%
%
%
\subsection{Numerical results}
Let us consider a mono-layer dielectric coating with characteristics $\varepsilon_r = 4.0$, $\mu_r = 1.0$ and $d = 0.005\lambda_0$. \Cref{TE_TM} shows values exact IBC, Leontovich IBC, first order and second order impedance boundary conditions, in TE polarization and TM polarization. Where the angle of incidence of the plane wave $\phi$ has angle range $]0, \pi[$.



\begin{figure}[H]
\begin{minipage}{0.45\textwidth}
\begin{figure}[H]
\centering
\includegraphics[scale=0.48]{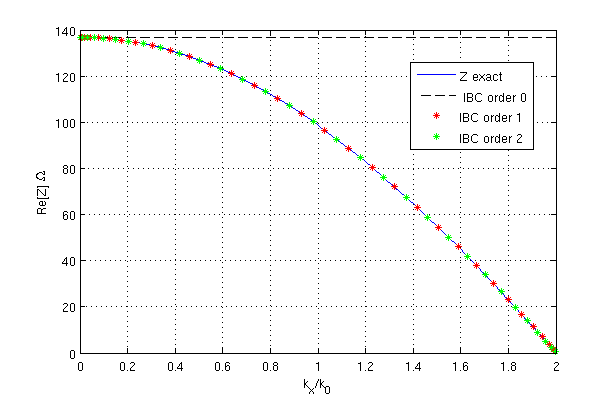}
 \subcaption{TE polarisation}
 \label{fig_ZTE}
\end{figure}
\end{minipage}   
\begin{minipage}{0.45\textwidth}
\begin{figure}[H]
\centering
\includegraphics[scale=0.48]{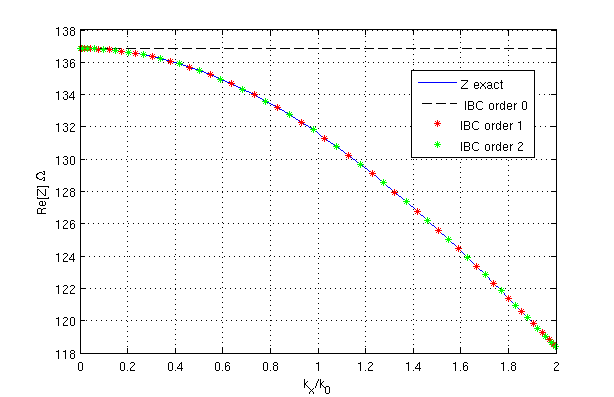}
\subcaption{TM polarisation}
\label{fig_ZTM}
\end{figure}
\end{minipage}
\caption{Comparison of the exact impedance, Leontovich impedance, first-order and second-order IBC.}
\label{TE_TM}
\end{figure}
We can easily see that the difference between IBC0 and exact IBC increases. While the difference between exact IBC and IBC1 is very small, as the difference between exact IBC and IBC2. But we can see the error of IBC1 and IBC2 approximations on the  \Cref{err_TE_TM}. As the angle of incidence increases the error of first-order IBC approximation reaches $0.39 \Omega$.
\begin{figure}[H]
\begin{minipage}{0.45\textwidth}
\begin{figure}[H]
\centering
\includegraphics[scale=0.48]{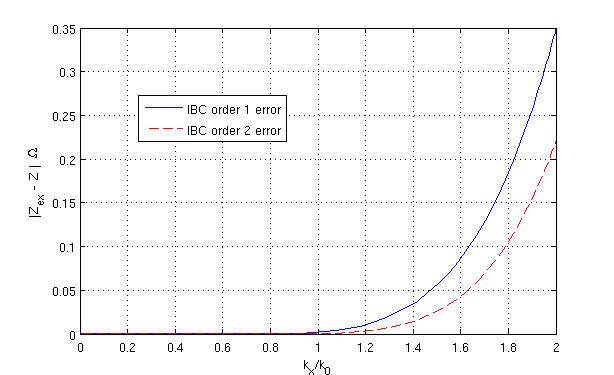}
\subcaption{TE polarisation}
\label{fig_difZTE}
\end{figure}
\end{minipage}   
\begin{minipage}{0.45\textwidth}
\begin{figure}[H]
\centering
\includegraphics[scale=0.48]{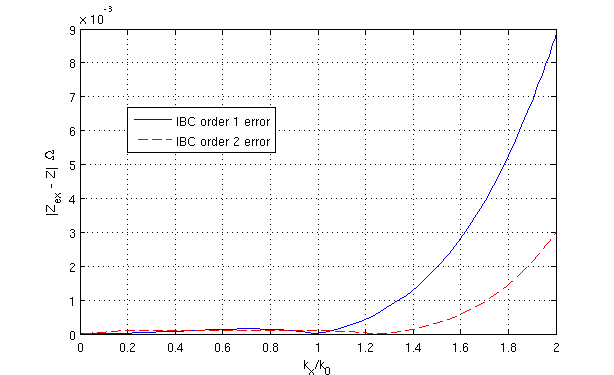}
\subcaption{TM polarisation}
\label{fig_difZTM}
\end{figure}
\end{minipage}
 \caption{Errors of first-order (IBC1) and second-order (IBC2) IBC.}
 \label{err_TE_TM}
\end{figure} 
In the next sections, we want to establish some  variational formulations, applying a boundary integral method for the problem in 2D case  with  different IBCs.
Then, we introduce current densities $\textbf{J} $ and $\textbf{M}$ on the boundary $\Gamma$ as follows: 
$$\mathbf{M} = \mathbf{E} \times \mathbf{n} \ \ , \ \mathbf{J} = \mathbf{n} \times \mathbf{H}.$$
Thanks to the Stratton-Chu formula, 
we know that $J$ and $M$ are solutions of EFIE and MFIE, as\cite{BEM}:
\begin{equation}{\label{EFIE}}
<Z_0(B-S) \mathbf{J}, \boldsymbol{\Psi}_J> + <(P+Q)\mathbf{M}, \boldsymbol{\Psi}_J> = <E^{inc}, \boldsymbol{\Psi}_J>,
\end{equation}
\begin{equation}{\label{MFIE}}
-<(P+Q)\mathbf{J}, \boldsymbol{\Psi}_M>  + <\frac{1}{Z_0}(B-S)\mathbf{M}, \boldsymbol{\Psi}_M> = <H^{inc}, \boldsymbol{\Psi}_M>
\end{equation}
\section{Variational formulation for the 2D problem}
%
%
Two dimensional case system is invariant in one direction, In the two dimensional case, 
so object surface $\Gamma$ becomes a curved contour, that we will call $C$. We have the curvilinear abscissa $l$ along $C$ and normal to the contour unit vector $\mathbf{n}$. We set the local frame ($\boldsymbol{\tau},\ \boldsymbol{\nu},\ \mathbf{n}$), where $\boldsymbol{\tau}$ is a unit vector tangent to the contour $C$ in $l$ direction, and $\boldsymbol{\nu}$ can be defined as $\boldsymbol{\nu} = \mathbf{n} \times \boldsymbol{\tau}$. We suppose that our two dimensional system does not depend on $\nu$ parameter, however variable $\nu$ component is depend on $l$. 
%
%
%
%
%
%


%
So, we write the first order IBC (\ref{Z_2Dj_boxed}) as follows  
\begin{equation}\label{2D_ibc1_eq}
(1 + b_j d_l^2)(\mathbf{n} \times \mathbf{M} ) = (a_0 + a_j d_l^2)\mathbf{J},
\end{equation}
%
and we can write the second order IBC (\ref{i_eq}) as follows
\begin{equation}\label{2D_ibc2_eq}
(1 + b_j d_l^2  + b'_j d_l^4) (\mathbf{n} \times \mathbf{M}) = (a_0 + a_j d_l^2 + a'_j d_l^4 ) \mathbf{J}, 
\end{equation}
where $j = 1,2$ correspond to TE and TM polarizations, respectively. The invariance in one direction for two dimensional model allow us to simplify the Hodge operator as the second partial derivative on the contour. Where the electromagnetic current densities $\mathbf{n}\times \mathbf{M}$ and $\mathbf{J}$ have $\boldsymbol{\tau}$ direction for TE polarization, and $\boldsymbol{\nu}$ direction for TM polarization.
\begin{itemize}
 \item TE: $\partial_x^2 \mathbf{J} = \boldsymbol{\tau} \partial_x^2 J_{\tau} = \boldsymbol{\tau} d_l^2 J_{\tau}$ and $\partial_x^2(\mathbf{n} \times \mathbf{M}) = - \boldsymbol{\tau} \partial_x^2 M_{\nu} = - \boldsymbol{\tau} d_l^2 M_{\nu}$;
 \item TM: $\partial_x^2 \mathbf{J} = \boldsymbol{\nu} \partial_x^2 J_{\nu} = \boldsymbol{\nu} d_l^2 J_{\nu}$ and $\partial_x^2(\mathbf{n} \times \mathbf{M}) = \boldsymbol{\nu} \partial_x^2 M_{\tau}  = \boldsymbol{\nu} d_l^2 M_{\tau}$.
\end{itemize}
%


\subsection{Variational formulation with first order IBC} 

We weakly write the equation (\ref{2D_ibc1_eq}), we multiply it by a test function and integrate along the contour $C$. In EFIE test function is noted as $\Psi_j$, so we use it to insert our boundary condition in this equation.
$$ \int_C (1 + b_j d_l^2)(\mathbf{n} \times \mathbf{M}) \cdot \boldsymbol{\Psi}_J dl = \int_C (a_0 + a_j d_l^2) \mathbf{J} \cdot \boldsymbol{\Psi}_J dl .$$
That gives us
$$ \int_C (\mathbf{n} \times \mathbf{M}) \cdot \boldsymbol{\Psi}_J dl =  \int_C (a_0 + a_j d_l^2)\mathbf{J} \cdot \boldsymbol{\Psi}_J dl -  \int_C b_j d_l^2 (\mathbf{n} \times \mathbf{M}) \cdot \boldsymbol{\Psi}_J dl .$$
We put it in the operator $P$ and obtain:
$$<P\mathbf{M}, \boldsymbol{\Psi}_J> = \frac{1}{2}\int_C (\mathbf{n} \times \mathbf{M}) \cdot \boldsymbol{\Psi}_J dl = \frac{a_0}{2}\int_C \mathbf{J} \cdot \boldsymbol{\Psi}_J dl $$
\begin{equation}\label{P_EFIE}
+ \frac{a_j}{2}\int_C d_l^2 \mathbf{J} \cdot \boldsymbol{\Psi}_J dl - \frac{b_j}{2}\int_C d_l^2 (\mathbf{n} \times \mathbf{M}) \cdot \boldsymbol{\Psi}_J dl
\end{equation}
We use $(\mathbf{n} \times \boldsymbol{\Psi}_M)$ as a test function, to insert IBC1 in MFIE
$$ \int_C (1 + b_j d^2_l)(\mathbf{n} \times \mathbf{M}) \cdot (\mathbf{n} \times \boldsymbol{\Psi}_M) dl = \int_C ( a_0 + a_j d^2_l)\mathbf{J} \cdot (\mathbf{n} \times \boldsymbol{\Psi}_M) dl .$$
We take the first part of right side
$$ \int_C \mathbf{J} \cdot (\mathbf{n} \times \boldsymbol{\Psi}_M ) dl = \frac{1}{a_0} \int_C (1 + b_j d^2_l)(\mathbf{n} \times \mathbf{M}) \cdot (\mathbf{n} \times \boldsymbol{\Psi}_M) dl - \frac{1}{a_0} \int_C a_j d^2_l \mathbf{J} \cdot (\mathbf{n} \times \boldsymbol{\Psi}_M ) dl .$$
And using the formula of vector analysis
$$\boldsymbol{\Psi}_M \cdot (\mathbf{n} \times \mathbf{J}) = - \mathbf{J} \cdot (\mathbf{n} \times \boldsymbol{\Psi}_M) ,$$
we put it in P operator with weakly form of IBC1
$$<P \mathbf{J}, \boldsymbol{\Psi}_M> = \frac{1}{2}\int_C (\mathbf{n} \times \mathbf{J}) \cdot \boldsymbol{\Psi}_M dl = -\frac{1}{2}\int_C \mathbf{J}\cdot (\mathbf{n} \times \boldsymbol{\Psi}_M) dl =$$
$$= -\frac{1}{2 a_0}\int_C (\mathbf{n} \times \mathbf{M}) \cdot (\mathbf{n} \times \boldsymbol{\Psi}_M) dl - \frac{b_j}{2 a_0 }\int_C d_l^2 (\mathbf{n} \times \mathbf{M}) \cdot (\mathbf{n} \times \boldsymbol{\Psi}_M) dl $$
\begin{equation}\label{P_MFIE}
+ \frac{a_j}{2 a_0} \int_C d_l^2 \mathbf{J} \cdot (\mathbf{n} \times \boldsymbol{\Psi}_M ) dl
\end{equation}

\vspace{2mm}

First, we observe TE polarization, where $P$ operator becomes:
$$\int_C P M_{\nu}\ \Psi_{J\tau} dl = \frac{a_0}{2} \int_C J_{\tau}\ \Psi_{J\tau} dl + \frac{a_1}{2} \int_C d_l^2 J_{\tau}\ \Psi_{J\tau} dl  + \frac{b_1}{2} \int_C d_l^2 M_{\nu}\ \Psi_{J\tau} dl$$
and
$$\int_C P J_{\tau}\ \Psi_{M\nu} dl = -\frac{1}{2 a_0}\int_C M_{\nu}\ \Psi_{M\nu} dl - \frac{b_1}{2 a_0}\int_C d_l^2 M_{\nu}\ \Psi_{M \nu} dl - \frac{a_1}{2 a_0}\int_C d_l^2 J_{\tau}\ \Psi_{M\nu} dl$$
for EFIE and MFIE, respectively.

We put them in the variational equations (\ref{EFIE}) and (\ref{MFIE}) and get:
$$ iZ_0 \iint_C k G(l,l')\ J_{\tau}(l')\ \Psi_{J\tau}\ [\boldsymbol{\tau}(l') \cdot \boldsymbol{\tau}(l)] - \frac{1}{k} G(l,l')\ d'_l J_{\tau}(l')\ d_l\Psi_{J\tau}(l)\ dl' dl$$
$$ + \iint_C \Psi_{J\tau}(l)\ M_{\nu}(l')\ [ \boldsymbol{\tau}(l) \times \boldsymbol{\nu}(l')] \cdot \nabla_{\Gamma} G(l,l')\ dl' dl $$
$$+ \frac{a_0}{2}\int_C J_{\tau}\ \Psi_{J\tau} dl + \frac{a_1}{2} \int_C d_l^2 J_{\tau}\ \Psi_{J\tau} dl $$
\begin{equation}\label{EFIE_var}
 + \frac{b_1}{2} \int_C d_l^2 M_{\nu}\ \Psi_{J\tau} dl = \int_C E_{\tau}^{inc}\ \Psi_{J\tau} dl
\end{equation}

and

$$-\iint_C \Psi_{M\nu}(l)\ J_{\tau}(l')\ [\boldsymbol{\nu(l)} \times \boldsymbol{\tau(l')}] \cdot \nabla_{\Gamma} G(l,l') dl' dl$$
$$ + \frac{i}{Z_0}\iint_C kG(l,l')\ M_{\nu}(l')\ \Psi_{M\nu}(l) [\boldsymbol{\nu(l')} \cdot \boldsymbol{\nu(l)}] dl' dl$$
$$+ \frac{1}{2 a_0}\int_C M_{\nu}\ \Psi_{M\nu} dl + \frac{b_j}{2 a_0} \int_C d_l^2 M_{\nu}\ \Psi_{M\nu} dl $$
\begin{equation}\label{MFIE_var}
 +\frac{a_j}{2 a_0} \int_C d_l^2 J_{\tau}\ \Psi_{M\nu} dl = \int_C H_{\nu}^{inc}\ \Psi_{M \nu} dl
\end{equation}
for EFIE and MFIE, respectively.

In the equations (\ref{EFIE_var}) and (\ref{MFIE_var}) we have scalar products $[ \boldsymbol{\tau}(l') \cdot \boldsymbol{\tau}(l) ] = 1$ and $[ \boldsymbol{\nu}(l) \cdot \boldsymbol{\nu}(l') ] = 1$, and vector products $[ \boldsymbol{\tau}(l) \times \boldsymbol{\nu}(l')] = \mathbf{n}(l)$ and $[ \boldsymbol{\nu}(l) \times \boldsymbol{\tau}(l') ] = -\mathbf{n}(l')$. The operator $S$ contains surface divergence operator that becomes differential operator
$$\div_{\Gamma}\mathbf{J} = \div_{\Gamma}(\boldsymbol{\tau} J_{\tau}) = d_l J_{\tau} ; \,\,\,\,\,\,\, \div_{\Gamma}\mathbf{M} = \div_{\Gamma}(\boldsymbol{\nu} M_{\nu}) = d_{\nu} M_{\nu} \equiv 0 ,$$
because the model is invariance in $\nu$ parameter.

By doing integration by parts, we have
\begin{equation}
 \frac{b_1}{2} \int_C d_l^2 M_{\nu}(l) \Psi_{J\tau}(l) dl = -\frac{b_1}{2} \int_C d_l M_{\nu}(l)\ d_l \Psi_{J\tau}(l) dl .
\end{equation}

Finally we combine two equations (\ref{EFIE_var})-(\ref{MFIE_var}) to present next variational problem:

\begin{problem}\label{Prb_Var_IBC1_2D} 
Find $U = (J_{\tau}, M_{\nu}) \in  [ H^1(C) ]^2$ such that:

\begin{equation*}\label{EqVar}
 A(U, \Psi) = \int_C E_{\tau}^{inc} \Psi_{J\tau} dl + \int_C H_{\nu}^{inc} \Psi_{M\nu} dl
\end{equation*}
for all $\Psi = (\Psi_{J\tau}, \Psi_{M\nu}) \in [ H^1(C) ]^2$, where the bilinear form $A$ is defined as:

$$ A(U, \Psi) = iZ_0 \iint_C k G(l,l') J_{\tau}(l') \Psi_{J\tau}\ [\tau(l)\cdot\tau(l')] - \frac{1}{k} G(l,l') d'_l J_{\tau}(l') d_l \Psi_{J\tau}(l) dl' dl$$
$$ + \iint_C \Psi_{J\tau}(l) M_{\nu}(l')\ \mathbf{n}(l) \cdot \nabla_{\Gamma} G(l,l') dl' dl + \iint_C \Psi_{M \nu} J_{\tau}\ \mathbf{n}(l') \cdot \nabla_{\Gamma} G(l,l') dl' dl $$
$$ + \frac{i}{Z_0}\iint_C kG(l,l') M_{\nu}(l') \Psi_{M\nu}(l) dl' dl + \frac{a_0}{2}\int_C J_{\tau} \Psi_{J\tau} dl + \frac{1}{2a_0}\int_C M_{\nu} \Psi_{M\nu} dl $$
\begin{equation}\label{VarForm}
- \frac{a_1}{2 }\int_C d_l J\ d_l\Psi_{J\tau} dl - \frac{b_1}{2} \int_C d_l M\ d_l\Psi_{J\tau} dl - \frac{b_1}{2 a_0} \int_C d_l M\ d_l\Psi_{M\nu} dl - \frac{a_1}{2 a_0} \int_C d_l J\ d_l\Psi_{M\nu} dl
\end{equation}
\end{problem}
%


We present similar variational problem for TM polarization:

\begin{problem} \label{problem_TM}
Find $U = (J_{\nu}, M_{\tau}) \in [ H^1(C) ]^2$ such that:
\begin{equation*}\label{EqVar_2}
 A(U, \Psi) = \int_C E_{\nu}^{inc} \Psi_{J\nu} dl + \int_C H_{\tau}^{inc} \Psi_{M\tau} dl
\end{equation*}
for all $\Psi = (\Psi_{J\nu}, \Psi_{M\tau}) \in [ H^1(C) ]^2$, where the bilinear form $A$ is defined as:
$$ A(U, \Psi) = iZ_0 \iint_C k G(l,l') J_{\nu}(l')\Psi_{J \nu} dl' dl - \iint_C \Psi_{J\nu}(l) M_{\tau}(l')\ \mathbf{n}(l') \cdot \nabla_{\Gamma} G(l,l') dl' dl $$
$$ - \iint_C \Psi_{M\tau} J_{\nu}\ \mathbf{n}(l) \cdot \nabla_{\Gamma} G(l,l') dl' dl + \frac{i}{Z_0}\iint_C kG(l,l') M_{\tau}(l') \Psi_{M\tau}(l) [\tau(l)\cdot\tau(l')] $$
$$- \frac{1}{k} G(l,l') d'_l M_{\tau}(l') d_l \Psi_{M \tau}(l)  dl' dl + \frac{a_0}{2}\int_C J_{\nu} \Psi_{J \nu} dl - \frac{1}{2 a_0}\int_C M_{\tau} \Psi_{M\tau} dl $$ 
\begin{equation}
- \frac{a_2}{2} \int_C d_l J\ d_l\Psi_{J \nu} dl  + \frac{b_2}{2} \int_C d_l M\ d_l\Psi_{J\nu} dl + \frac{b_2}{2 a_0} \int_C d_l M\ d_l\Psi_{M \tau} dl - \frac{a_2}{2 a_0} \int_C d_l J\ d_l\Psi_{M\tau} dl
\end{equation}
\end{problem}
\subsection{Variational formulation with second order IBC} 

The equation (\ref{2D_ibc2_eq}) passes the same way as IBC1 to become weak. The weak formulations replace operator $P$ in EFIE and MFIE equations. Finally, we assemble them to define the bilinear form:
$$ A(U, \Psi) = iZ_0 \iint_C k G(l,l') J_{\tau}(l') \Psi_{J\tau}\ [\boldsymbol{\tau}(l) \cdot \boldsymbol{\tau}(l')] - \frac{1}{k} G(l,l') d'_l J_{\tau}(l') d_l \Psi_{J\tau}(l) dl' dl$$
$$ + \iint_C \Psi_{J\tau}(l) M_{\nu}(l')\ \mathbf{n}(l) \cdot \nabla_{\Gamma} G(l,l') dl' dl
 + \iint_C \Psi_{M\nu} J_{\tau}\ \mathbf{n}(l') \cdot \nabla_{\Gamma} G(l,l') dl' dl $$
$$ + \frac{i}{Z_0} \iint_C kG(l,l') M_{\nu}(l') \Psi_{M\nu}(l) dl' dl 
 + \frac{a_0}{2} \int_C J_{\tau} \Psi_{J\tau} dl + \frac{1}{2 a_0}\int_C M_{\nu} \Psi_{M\nu} dl $$
$$ + \frac{a_1}{2} \int_C d^2_l J_{\tau} \Psi_{J\tau} dl + \frac{b_1}{2} \int_C d^2_l M_{\nu} \Psi_{J\tau} dl
 + \frac{b_1}{2 a_0} \int_C d^2_l M_{\nu} \Psi_{M\nu} dl + \frac{a_1}{2 a_0} \int_C d^2_l J_{\tau} \Psi_{M\nu} dl$$
$$ + \frac{a'_1}{2}\int_C d^4_l M_{\nu}\ \Psi_{J\tau} dl + \frac{b'_1}{2} \int_C d^4_l M_{\nu}\ \Psi_{J\tau} + \frac{b'_1}{2 a_0} \int_C d^4_l M_{\nu}\ \Psi_{M\nu} dl + \frac{a'_1}{2 a_0} \int_C d^4_l J_{\tau}\ \Psi_{M\nu} dl $$
for TE polarization. And with integration by parts, we get
$$ A(U, \Psi) = iZ_0 \iint_C k G(l,l') J_{\tau}(l') \Psi_{J\tau}\ [\boldsymbol{\tau}(l) \cdot \boldsymbol{\tau}(l')] - \frac{1}{k} G(l,l') d'_l J_{\tau}(l') d_l \Psi_{J\tau}(l) dl' dl$$
$$ + \iint_C \Psi_{J\tau}(l) M_{\nu}(l')\ \mathbf{n}(l) \cdot \nabla_{\Gamma} G(l,l') dl' dl
 + \iint_C \Psi_{M\nu} J_{\tau}\ \mathbf{n}(l') \cdot \nabla_{\Gamma} G(l,l') dl' dl $$
$$ + \frac{i}{Z_0} \iint_C kG(l,l') M_{\nu}(l') \Psi_{M\nu}(l) dl' dl 
 + \frac{a_0}{2} \int_C J_{\tau} \Psi_{J\tau} dl + \frac{1}{2 a_0}\int_C M_{\nu} \Psi_{M\nu} dl $$
$$ - \frac{a_1}{2} \int_C d_l J_{\tau}\ d_l\Psi_{J\tau} dl - \frac{b_1}{2} \int_C d_l M_{\nu}\ d_l\Psi_{J\tau} dl
 - \frac{b_1}{2 a_0} \int_C d_l M_{\nu}\ d_l\Psi_{M\nu} dl - \frac{a_1}{2 a_0} \int_C d_l J_{\tau}\ d_l\Psi_{M\nu} dl$$
$$ + \frac{a'_1}{2}\int_C d^2_l M_{\nu}\ d^2_l\Psi_{J\tau} dl + \frac{b'_1}{2} \int_C d^2_l M_{\nu}\ d^2_l\Psi_{J\tau} 
+ \frac{b'_1}{2 a_0} \int_C d^2_l M_{\nu}\ d^2_l\Psi_{M\nu} dl + \frac{a'_1}{2 a_0} \int_C d^2_l J_{\tau}\ d^2_l\Psi_{M\nu} dl $$
We thus obtain a variational formulation for the problem in $\mathbb{R}^2$ with second order IBC written as follows:
\begin{problem}\label{problem_IBC2}
Find $U = (J_{\tau}, M_{\nu}) \in [ H^1(C) ]^2 $ such that 
$$A(U, \Psi) = \int_C E_{\tau}^{inc} \Psi_{J\tau} dl + \int_C H_{\nu}^{inc} \Psi_{M\nu} dl $$
for all $\Psi = ( \Psi_{J\tau}, \Psi_{M\nu} ) \in [ H^1(C) ]^2$.
\end{problem}

In the following, we study the well-posedness of the formulation (\ref{Prb_Var_IBC1_2D})
\subsection{Existence and uniqueness theorem}
We are going to show that our variational problem in TE has a unique solution using theorem  from \cite{JCN}. For TM problem \ref{problem_TM}, it will be analogous. It is necessary to determine the continuity and the coercivity of the bilinear form $A(U, \Psi)$.

For the sake of simplicity we consider the operator $A(U, \Psi)$ as a sum of three bilinear operator
$$A_1(U, \Psi) = \iint_C Z_0(B-S) J_{\tau} \Psi_{J\tau} dl'dl + \iint_C \frac{1}{Z_0} (B-S) M_{\nu} \Psi_{M\nu} dl'dl + \iint_C Q M_{\nu} \Psi_{J\tau} dl'dl$$
$$ + \iint_C Q J_{\tau} \Psi_{M\nu} dl'dl + \frac{a_0}{2} \int_C J_{\tau} \Psi_{J\tau} dl + \frac{1}{2 a_0} \int_C M_{\nu} \Psi_{M\nu} dl$$
$$A_2(U, \Psi) = -\frac{a_1}{2}\int_C d_l J\ d_l\Psi_{J\tau} dl - \frac{b_1}{2 a_0} \int_C d_l M\ d_l\Psi_{M\nu} dl$$
and
$$A_3(U, \Psi) = -\frac{b_1}{2} \int_C d_l M\ d_l\Psi_{J\tau} dl - \frac{a_1}{2 a_0} \int_C d_l J\ d_l\Psi_{M\nu} dl$$
where
$$A = A_1 + A_2 +A_3$$

%
We have the following result about continuous linear  form
$A(U, \Psi)$.
\begin{lem}\label{lemma_cont_2d}
 The bilinear form $A(U, \Psi)$ (\ref{VarForm}) is continuous on $[H^1(C)]^2$.
\end{lem}

\begin{proof}:
%
We have to show that there exists $\beta > 0$ such that 
\begin{equation}\label{cont_ineq}
 |A(U, \Psi)| \leq \beta \|U\|_{H^1(C)} \| \Psi \|_{H^1(C)} 
\end{equation}
for all $U, \Psi \in H^1(C)$.

From the works in the past \cite{Lange-1995}, we have that for operator $A_1(U, \Psi)$ 
there exists constant $\beta_1 > 0$ such that
%
%
$$|A_1(U, \Psi)| \leq \beta_1  \| U \|_{H^1(C)} \| \Psi \|_{H^1(C)} $$
It remains to prove continuity for last integrals in bilinear form. Using Cauchy-Schwarz inequality we get:
$$ |A_2(U, \Psi) + A_3(U, \Psi)| \leq $$
$$\left| \frac{a_1}{2}\int_C d_l J\ d_l\Psi_{J\tau} dl \right| + \left| \frac{b_1}{2} \int_C d_l M\ d_l\Psi_{J\tau} dl \right| + \left| \frac{b_1}{2 a_0} \int_C d_l M\ d_l\Psi_{M\nu} dl \right| + \left| \frac{a_1}{2 a_0} \int_C d_l J\ d_l\Psi_{M\nu} dl \right| \leq $$
$$\left| \frac{a_1}{2}\right| \|d_l J\|_{L^2} \| \Psi_{J\tau}\|_{L^2} + \left| \frac{b_1}{2} \right| \|d_l M\|_{L^2} \| \Psi_{J\tau}\|_{L^2}  + \left| \frac{b_1}{2 a_0} \right| \|d_l M \|_{L^2} \| \Psi_{M\nu}\|_{L^2} + \left| \frac{a_1}{2 a_0} \right| \|d_l J \|_{L^2} \| \Psi_{M\nu} \|_{L^2}  $$
$$ \leq \beta_2 \| U \|_{H^1(C)} \| \Psi \|_{H^1(C)} \ \ ,\ where\ \beta_2\geq 0. $$
Finally, we take $\beta = \beta_1 + \beta_2 \geq 0$ for which (\ref{cont_ineq}) is true. 
\end{proof}

It remains to show  the existence and uniqueness for the variational problem \ref{Prb_Var_IBC1_2D}. Therefore, we give the following result about coercivity of the bilinear form $A(U, \Psi)$
\begin{lem}\label{lemma_coer_2d}
The bilinear form $A(U, \Psi)$ is coercive on $H^1(C)$; i.e., there exists $\gamma > 0$ and $\gamma'$ such that
$$ \Re[A(U, U^*) ] \geq \gamma \|U\|^2_{H^1(C)} - \gamma' \|U\|^2_{L^2(C)}, \ \ \ \forall U \in [H^1(C)]^4. $$ 
\end{lem}

\begin{proof}
First of all, we take $\Psi = U^*$ and get
$$ A(U, U^*) = iZ_0 \iint_C k G J_{\tau} J^*_{\tau} [\boldsymbol{\tau}(l') \cdot \boldsymbol{\tau}(l)] - \frac{1}{k} d'_l J_{\tau}(l') d_l J^*_{\tau}(l) dl' dl$$
$$ + \iint_C J^*_{\tau} M_{\nu} \mathbf{n}(l) \cdot \nabla_{\Gamma} G dl' dl 
 + \iint_C M^*_{\nu} J_{\tau} \mathbf{n}(l') \cdot \nabla_{\Gamma} G dl' dl $$
$$ + \frac{i}{Z_0} \iint_C kG M_{\nu} M^*_{\nu} [\boldsymbol{\nu}(l') \cdot \boldsymbol{\nu}(l)] dl' dl 
 + \frac{a_0}{2}\int_C J_{\tau} J^*_{\tau} dl + \frac{1}{2 a_0}\int_C M_{\nu} M^*_{\nu} dl $$
\begin{equation}\label{eq_coer_2d}
 - \frac{a_1}{2}\int_C d_l J_{\tau}\ d_l J^*_{\tau} dl - \frac{b_1}{2} \int_C d_l M_{\nu}\ d_l J^*_{\tau} dl 
 - \frac{b_1}{2 a_0} \int_C d_l M_{\nu}\ d_l M^*_{\nu} dl - \frac{a_1}{2 a_0} \int_C d_l J_{\tau}\ d_l M^*_{\nu} dl
\end{equation}

As in continuity, from works in the past \cite{Lange-1995}, we have that there exists $\gamma_1 > 0$ for 
the operator $A_1(U, \Psi)$, such that $\forall U \in V$
%
%
%
$$\Re[A_1(U, U^*)] \geq \frac{\Re(a_0)}{2} \|J_{\tau}\|^2_{L^2(C)} + \frac{\Re(a_0)}{2 |a_0|^2} \|M_{\nu}\|^2_{L^2(C)} + \gamma_1 \left( \|J_{\tau}\|^2_{H^1(C)} + \|M_{\nu}\|^2_{H^1(C)} \right)$$

Next we take operator $A_2$
$$A_2 = - \frac{a_1}{2}\int_C d_l J_{\tau}\ d_l J^*_{\tau} dl - \frac{b_1}{2 a_0} \int_C d_l M_{\nu}\ d_l M^*_{\nu} dl $$
where real part
$$\Re(A_2) = -\frac{ \Re(a_1) }{2} \| d_l J_{\tau} \|^2_{L^2(C)} -  \Re(\frac{b_1}{2 a_0}) \|d_l M_{\nu}\|^2_{L^2(C)} $$

And it remains
$$A_3 = - \frac{b_1}{2} \int_C d_l M_{\nu}\ d_l J^*_{\tau} dl - \frac{a_1}{2 a_0} \int_C d_l J_{\tau}\ d_l M^*_{\nu} dl $$
where real part
$$\Re(A_3) = \Re\left( -\frac{b_1}{2}\int_C d_l M_{\nu}\ d_l J^*_{\tau} dl - \frac{a_1}{2 a_0} \int_C d_l J_l\ d_l M^*_{\nu} dl \right) = $$
$$= - \Re\left[ \left( \frac{b_1}{2} + \frac{a^*_1 a_0}{2 |a_0|^2} \right) \int_C d_l M_{\nu}\ d_l J^*_{\tau} dl \right] =$$
$$= -\Re \left[ \int_C \frac{1}{|a_0|^{1/2}} \left( \frac{b_1}{2} + \frac{a^*_1 a_0}{2 |a_0|^2} \right)^{1/2} d_l M_{\nu} \cdot |a_0|^{1/2} \left( \frac{b_1}{2} + \frac{a^*_1 a_0}{2 |a_0|^2} \right)^{1/2} d_l J^*_l dl\right] $$
We note $q = b_1 |a_0| + a^*_1 a_0 /|a_0| $, so 
$$\Re(A_3) \geq - \frac{|q|}{4} \|d_l J_{\tau}\|^2_{L^2(C)} - \frac{|q|}{4|a_0|^2} \|d_l M_{\nu}\|^2_{L^2(C)} .$$

Sufficient uniqueness conditions (\ref{SUC_1}) says that $\Re(a_1 - b^*_1 a_0) = 0$ or $\Re(a_1) = \Re(b_1 a^*_0)$. So the sum of operators $A_2$ and $A_3$ get
$$\Re(A_2) + \Re(A_3) \geq $$
$$ - \frac{1}{2} \left( \Re(a_1) + \frac{|q|}{2} \right) \| d_l J_{\tau} \|^2_{L^2(C)} - \frac{1}{2 |a_0|^2} \left( \Re(a_1) + \frac{|q|}{2} \right) \| d_l M_{\nu}\|^2_{L^2(C)}$$
So, if $\Re(a_1) + \frac{|q|}{2} = 0$, we get that
$$\Re(A) = \Re(A_1) + \Re(A_2) + \Re(A_3) \geq $$
$$\geq \gamma_1 \|U\|^2_{H^1(C)} - c \|U\|^2_{L^2(C)} $$

%
%
%
%

%
%
That gives us coercivity of bilinear form $A(U, \Psi)$. 
\end{proof}
Therefore, we have the result.
\begin{theorem}
 The problem (\ref{EqVar}) admits a unique solution $U \in [H^1(C)]^2$ for any $\Psi \in [H^1(C)]^2$, if coefficients satisfy 
\begin{equation}\label{eq_ex&un_coef_2d}
 \Re(a_1) + \frac{ |a_0| | b_1 + a^*_1/a^*_0 |}{2} = 0 .
\end{equation}
\end{theorem}

\begin{proof}
Lemmas \ref{lemma_cont_2d}-\ref{lemma_coer_2d} give us that the bilinear form $A(U, \Psi)$ verifies hypothesis of the theorem 5.6.1 (see \cite{JCN}). 
\end{proof}
\subsection{Two dimensional discretization}
We again recognize that the two dimensional case is invariant in one direction and a two dimensional object has a unidimensional boundary. So we will use finite element method on a curvilinear contour. We use auxiliary variables $X, Y$ that were mentioned in \cite{Stupfel_APR_2005}, to avoid integration by parts. So we take formulation of bilinear form $A(U, \Psi)$ before carrying out integration by parts
%
%
%
%
%
%
%
%
And finally we get bilinear form for $U = (J_{\tau}, M_{\nu}, X, Y) \in [H^1(C)]^4$
$$A(U, \Psi) =  iZ_0 \iint_C k G(l,l')\ J_{\tau}(l')\ \Psi_{J\tau}\ [\boldsymbol{\tau}(l')\cdot\boldsymbol{\tau}(l)] - \frac{1}{k} G(l,l')\ d'_l J_{\tau}(l')\ d_l \Psi_{J\tau}(l) dl' dl$$
$$ + \iint_C \Psi_{J\tau}(l)\ M_{\nu}(l')\ \mathbf{n}(l) \cdot \nabla_{\Gamma} G(l,l') dl' dl
 + \iint_C \Psi_{M\nu}\ J_{\tau}\ \mathbf{n}(l') \cdot \nabla_{\Gamma} G(l,l') dl' dl $$
$$ + \frac{i}{Z_0}\iint_C k G(l,l')\ M_{\nu}(l')\ \Psi_{M\nu}(l) dl' dl
 + \frac{a_0}{2}\int_C J_{\tau}\ \Psi_{J\tau} dl + \frac{1}{2 a_0}\int_C M_{\nu}\ \Psi_{M\nu} dl $$
$$ + \frac{a_1}{2} \int_C d_l X\ \Psi_{J\tau} dl + \frac{b_1}{2} \int_C d_l Y\ \Psi_{J\tau} dl + \frac{b_1}{2 a_0} \int_C d_l Y\ \Psi_{M\nu} dl + \frac{a_1}{2 a_0} \int_C d_l X\ \Psi_{M\nu} dl$$
\begin{equation}\label{Dis_form_A}
+ c_1 \int_C X\ X' dl - c_1 \int_C d_l J_{\tau}\ X' dl + d_1\int_C Y\ Y' dl - d_1 \int_C d_l M_{\nu}\ Y' dl
\end{equation}
for all $\Psi = (\Psi_{J\tau}, \Psi_{M\nu}, X', Y') \in [H^1(C)]^4$.

Here, we study only TE polarization case. We assume that TM polarization has analogous results.


We approximate the curve $C$ by means of $N$ straight \textit{line segments} $C_i$, satisfying the general overlapping conditions for a finite element method. The line segments are also called \textit{elements} and provide a contour piecewise linear approximation $C^h = \cup^N_{i = 1} C_i$, where $h$ is a positive parameter such that $\lim_{h\rightarrow 0} N(h) = +\infty$. Curvilinear structures include geometry modeling errors in this approach. These errors can only be reduced by decreasing the segments lengths; i.e. by increasing the number of segments $N$. We denote nodes from $1$ to $N$. 
We consider $V_h$ a finite dimensional subspace of $H^1(C^h)$ 
$$V_h = \left\{ v_h: C^h \rightarrow \R,\ v_h \in H^1(C^h),\ v_h|_{C_i} \in P_1,\ \forall i\in {1, ... , N} \right\} \subset H^1(C^h)$$
where $P_1$ is the space of first degree polynomials, and
$$W_h = \left\{ w_h: C^h \rightarrow \R,\ w_h \in H^1(C^h),\ w_h|_{C_i} \in P_0,\ \forall i\in {1, ... , N} \right\} \subset L^2(C^h)$$
where $P_0$ is the space of constant functions.

We observe discretization of the unknowns by basis functions
\begin{equation}\label{eq_sum_basis_func_J}
 J_{\tau} \approx J^h_{\tau}(l) = \sum^N_{i=1} J_{\tau i} \phi_i(l) \ \in V_h
\end{equation}
\begin{equation}\label{eq_sum_basis_func_M}
 M_{\nu} \approx M^h_{\nu}(l) = \sum^N_{i=1} M_{\nu i} \psi_i(l) \ \in W_h
\end{equation}
\begin{equation}\label{eq_sum_basis_func_X}
 X \approx X^h(l) = \sum^N_{i=1} X_i \psi_i(l) \ \in W_h
\end{equation}
and
\begin{equation}\label{eq_sum_basis_func_Y}
 Y \approx Y^h(l) = \sum^N_{i=1} Y_i \psi_i(l) \ \in W_h .
\end{equation}
where $\phi_i \in V_h$ and $\psi_i \in W_h$

So the bilinear form $A(U, \Psi)$ in (\ref{Dis_form_A}) can be written as
$$A(U^h, \Psi^h) = iZ_0 \sum_{i,j = 1}^N \left( \iint_{C^h} kG\ \phi_j\ \phi_i\ [\vec{\tau}_j(l')\cdot\vec{\tau}_i(l)] - \frac{1}{k}G\ d'_l \phi_j\ d_l \phi_i\ dl'dl \right) J^h_{\tau j} $$
$$+ \sum_{i,j = 1}^N \left( \iint_{C^h} \psi_j\ \phi_i\ \mathbf{n}_i\cdot\nabla_{\Gamma} G dl'dl \right) M^h_{\nu j} $$
$$+ \sum_{i,j = 1}^N \left( \iint_{C^h} \phi_j\ \psi_i\ \mathbf{n}_j\cdot\nabla_{\Gamma} G dl'dl \right) J^h_{\tau j}
+ \frac{i}{Z_0} \sum_{i,j = 1}^N \left( \iint_{C^h} kG\ \psi_j\ \psi_i\ dl'dl \right) M^h_{\nu j} $$
$$+ \frac{a_0}{2} \sum_{i,j = 1}^N \left( \int_{C^h} \phi_j\ \phi_i dl \right) J^h_{\tau j} + \frac{1}{2 a_0} \sum_{i,j = 1}^N \left( \int_{C^h} \psi_j\ \psi_i dl \right) M^h_{\nu j} $$
$$ + \frac{a_1}{2} \sum_{i,j = 1}^N \left( \int_{C^h} d_l\psi_j\ \phi_i dl \right) X^h_j + \frac{b_1}{2} \sum_{i,j = 1}^N \left( \int_{C^h} d_l\psi_j\ \phi_i dl \right) Y^h_j $$
$$+ \frac{b_1}{2 a_0} \sum_{i,j = 1}^N \left( \int_{C^h} d_l\psi_j\ \psi_i dl \right) Y^h_j + \frac{a_1}{2 a_0} \sum_{i,j = 1}^N \left( \int_{C^h} d_l\psi_j\ \psi_i dl \right) X^h_j $$
$$+ \sum_{i,j = 1}^N \left( \int_{C^h} \psi_j\ \psi_i dl \right) X^h_j - \sum_{i,j = 1}^N \left( \int_{C^h} d_l\psi_j\ \psi_i dl \right) J^h_{\tau j} $$
$$+ \sum_{i,j = 1}^N \left( \int_{C^h} \psi_j\ \psi_i dl \right) Y^h_j - \sum_{i,j = 1}^N \left( \int_{C^h} d_l\psi_j\ \psi_i dl \right) M^h_{\nu j} $$

In order to simplify the equations in matrix form we define the following matrices
$$(B-S)_{ij} = i \iint_{C^h} k G(l,l')\ \phi_j(l')\ \phi_i\ [\boldsymbol{\tau}_j\cdot\boldsymbol{\tau}_i] - \frac{1}{k} G(l,l')\ d'_l \phi_j(l')\ d_l \phi_i(l)\ dl' dl$$
$$Q_{ij} = \iint_{C^h} \phi_i(l)\ \psi_j(l')\ \mathbf{n}_i \cdot \nabla_{\Gamma} G(l,l')\ dl' dl $$
$$B_{ij} = i \iint_{C^h} k G(l,l')\ \psi_j(l')\ \psi_i(l)\ dl' dl $$
$$I1_{ij} = \int_{C^h} \phi_i(l)\ \phi_j(l)\ dl $$
$$I2_{ij} = \int_{C^h} \psi_i(l)\ \psi_j(l)\ dl $$
%
%
$$D1_{ij} = \int_{C^h} \phi_i(l)\ d_l\psi_j(l)\ dl $$
$$D3_{ij} = \int_{C^h} \psi_i(l)\ d_l\psi_j(l)\ dl $$
$$D5_{ij} = \int_{C^h} \psi_i(l)\ d_l\phi_j(l)\ dl $$
\textbf{Note}: $[I2]$ is the nonsingular diagonal matrix, therefore it is invertible.

So we can write our system in a matrix form:
\begin{equation}\label{eq_2d_matrix_1}
%
\left[
\begin{array}{cccc}
Z_0[B-S] + \frac{a_0}{2} [I1] & [Q] & \frac{a_1}{2}[D1] & \frac{b_1}{2}[D1] \\
{[Q]}^T & \frac{1}{Z_0} [B] + \frac{1}{2 a_0} [I2] & \frac{a_1}{2 a_0}[D3] & \frac{b_1}{2 a_0}[D3] \\
-[D5] & 0 & [I2] & 0 \\
0 & -[D3] & 0 & [I2]
\end{array}
\right]
\left (
\begin{array}{cccc}
\overline{J}^h \\ \overline{M}^h \\ \overline{X}^h \\ \overline{Y}^h
\end{array}
\right) = 
\left (
\begin{array}{cccc}
\overline{E}^h \\ \overline{H}^h \\ 0 \\ 0
\end{array}
\right)
%
\end{equation}
Where right-side vectors $\overline{E}^h$, $\overline{H}^h$ are defined as follows:
$$E_i^h = \int_{C^h} \mathbf{E}^{inc}\cdot\boldsymbol{\phi}_i dl ;\,\,\,\,\,\,\,\,\, H_i^h = \int_{C^h} \mathbf{H}^{inc}\cdot\boldsymbol{\psi}_i dl .$$
And vectors $\overline{J}^h$ and $\overline{M}^h$ are unknowns.
Now, we are interested in elimination of auxiliary variable.
From the last two lines in (\ref{eq_2d_matrix_1}), we get
$$- [D5]\ \overline{J}^h + [I2]\ \overline{X}^h = 0\ \rightarrow\ \overline{X}^h = {[I2]}^{-1}\ [D5]\ \overline{J}^h ;$$
$$- [D3]\ \overline{M}^h + [I2]\ \overline{Y}^h = 0\ \rightarrow\ \overline{Y}^h =  {[I2]}^{-1}\ [D3]\ \overline{M}^h .$$

We put $X, Y$ in first two equations.

\begin{equation}\label{eq_2d_matrix_2}
\left[
\begin{array}{cc}
[A1] & [A2] \\
{[A3]} & [A4]
\end{array}
\right]
\left(
\begin{array}{cc}
\overline{J}^h \\ \overline{M}^h 
\end{array}
\right) = 
\left(
\begin{array}{cc}
\overline{E}^h \\ \overline{H}^h 
\end{array}
\right)
\end{equation}
where matrices are defined as

$$[A1] = Z_0[B-S] + \frac{a_0}{2} [I1] + \frac{a_1}{2} [D1]\ {[I2]}^{-1}\ [D5] $$
$$[A2] = [Q] + \frac{b_1}{2}[D1]\ {[I2]}^{-1}\ [D3] $$
$$[A3] = {[Q]}^T + \frac{a_1}{2 a_0}[D3]\ {[I2]}^{-1}\ [D5] $$
$$[A4] = \frac{1}{Z_0} [B] + \frac{1}{2 a_0} [I2] + \frac{b_1}{2 a_0}[D3]\ {[I2]}^{-1}\ [D3] .$$

\subsection{Numerical tests 2D}

\begin{minipage}{0.45\textwidth}
\begin{figure}[H]
\centering
        \includegraphics[scale=0.31]{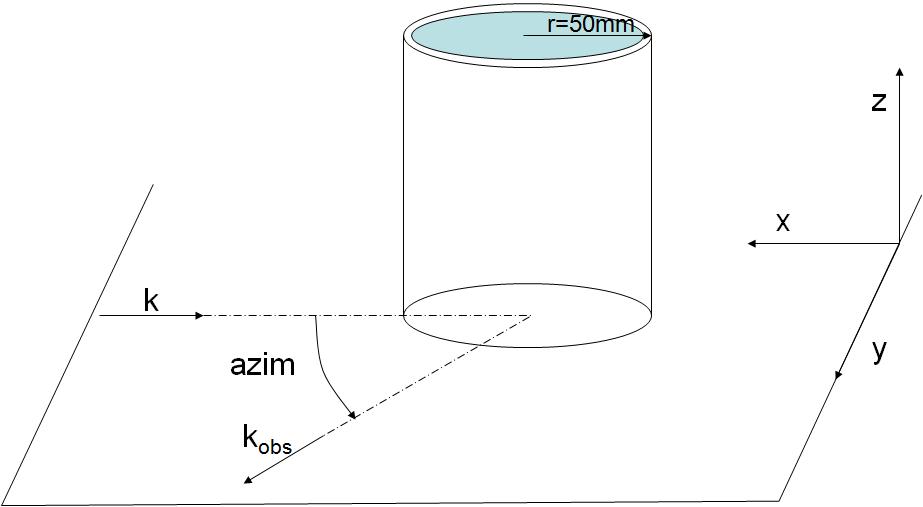}
        \caption{Cylinder}    
\label{fig:Cylinder}
    \end{figure}
\end{minipage}   
\begin{minipage}{0.45\textwidth}
\begin{figure}[H]
\centering
\includegraphics[scale=0.33]{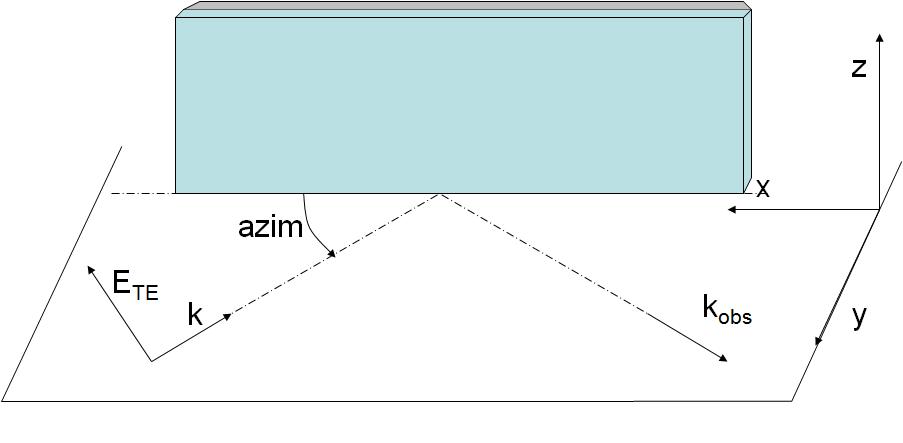}
 \caption{Plate with thin layer}
  \label{fig:2Dplate}
\end{figure}
\end{minipage}

In order to illustrate these points scattering by three typical coated cylinders (see \Cref{fig:Cylinder}) will be considered next.  \Cref{cyl_ep5_TE_TM} show the monostatic RCS 
for a coated conducting cylinder with inner radius $\lambda_0$, coating thickness $d = 0.1\lambda_0$, and coating parameters $\epsilon_r = 4.0 - 0.5i$ and $\mu_r = 1.0$. The exact series solution is presented along with the HOIBC and SIBC solutions. We computed monostatic RCS for different frequencies to see how do results depend on frequency. In TE-polarization 
we can see that results of SIBC jumps in range between $6GHz$ and $8GHz$ (see \Cref{cyl_ep5_TE}). Much bigger difference, we can see in TM-polarization 
between $7GHz$ and $9GHz$ (see fig. {\ref{cyl_ep5_TM}}).
\begin{figure}[H]
\begin{minipage}{0.45\textwidth}
\begin{figure}[H]
\centering
        \includegraphics[scale=0.52]{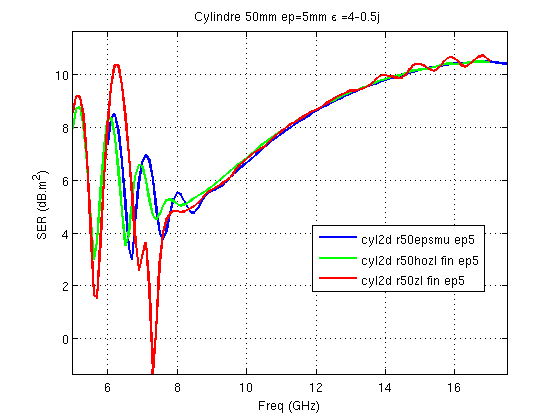}
\subcaption{TE polarization}
\end{figure}
\end{minipage}   
\begin{minipage}{0.45\textwidth}
\begin{figure}[H]
\centering
\includegraphics[scale=0.5]{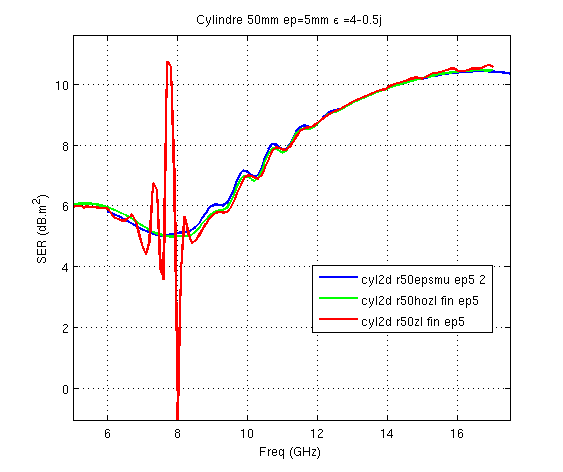}
 \subcaption{TM polarization}
\end{figure}
\end{minipage}
 \caption{Monostatic RCS for a coated circular cylinder} 
 \label{cyl_ep5_TE_TM}
\end{figure}
Next, we consider bistatic RCS for different scattering angles. \Cref{cyl_ep1p5_TM_TE} show the bistatic radar cross section for a coated conducting cylinder with inner radius $\lambda_0$, coating thickness $d = 1.5mm$, and coating parameters $\epsilon_r = 10 - 5i$ and $\mu_r = 1$, for fixed frequency $f = 6.8GHz$ in TE and TM polarizations. The exact series solution is presented along with the IBC0 and HOIBC order 1 and order 2 solutions.

\begin{figure}[H]
\begin{minipage}{0.45\textwidth}
\begin{figure}[H]
\centering
\includegraphics[scale=0.4]{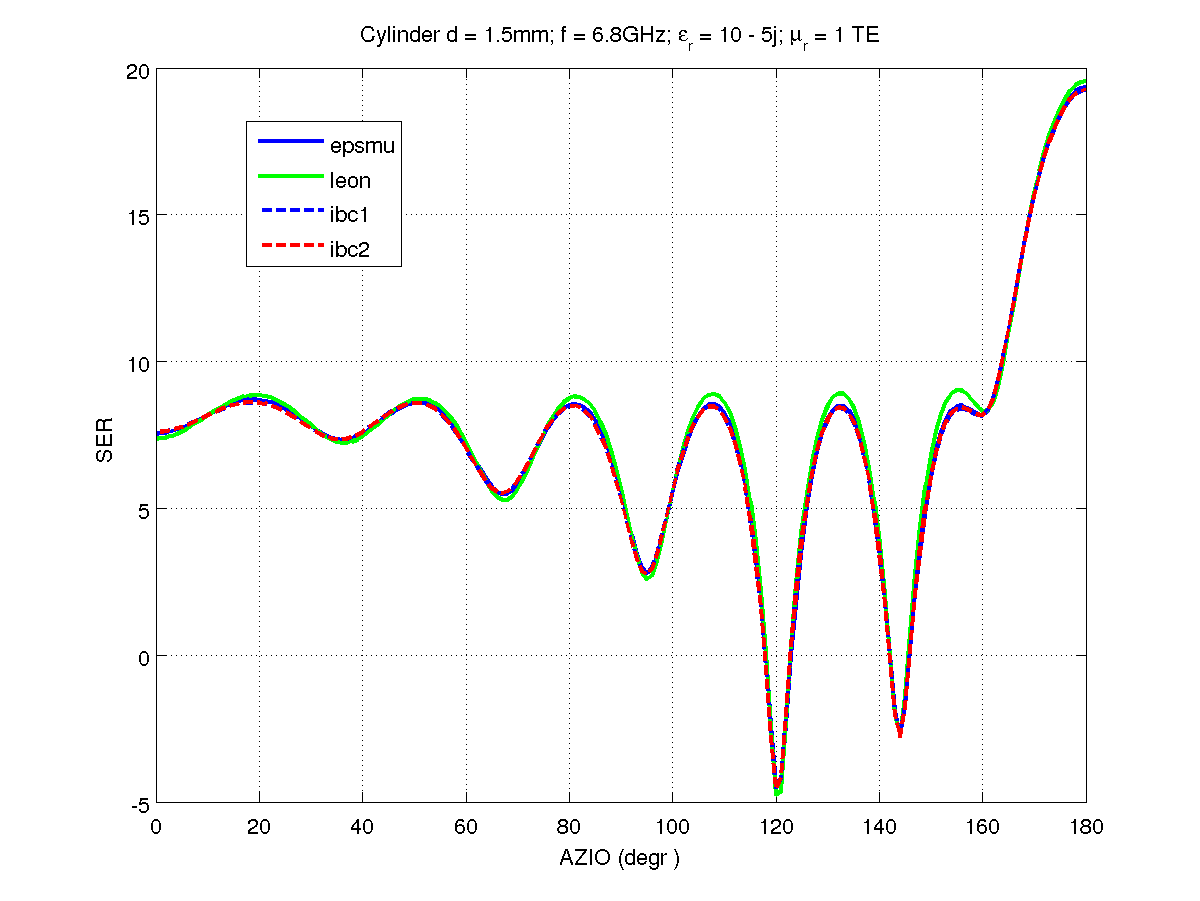}
 \subcaption{TE polarization}
\end{figure}
\end{minipage}   
\begin{minipage}{0.45\textwidth}
\begin{figure}[H]
\centering
\includegraphics[scale=0.4]{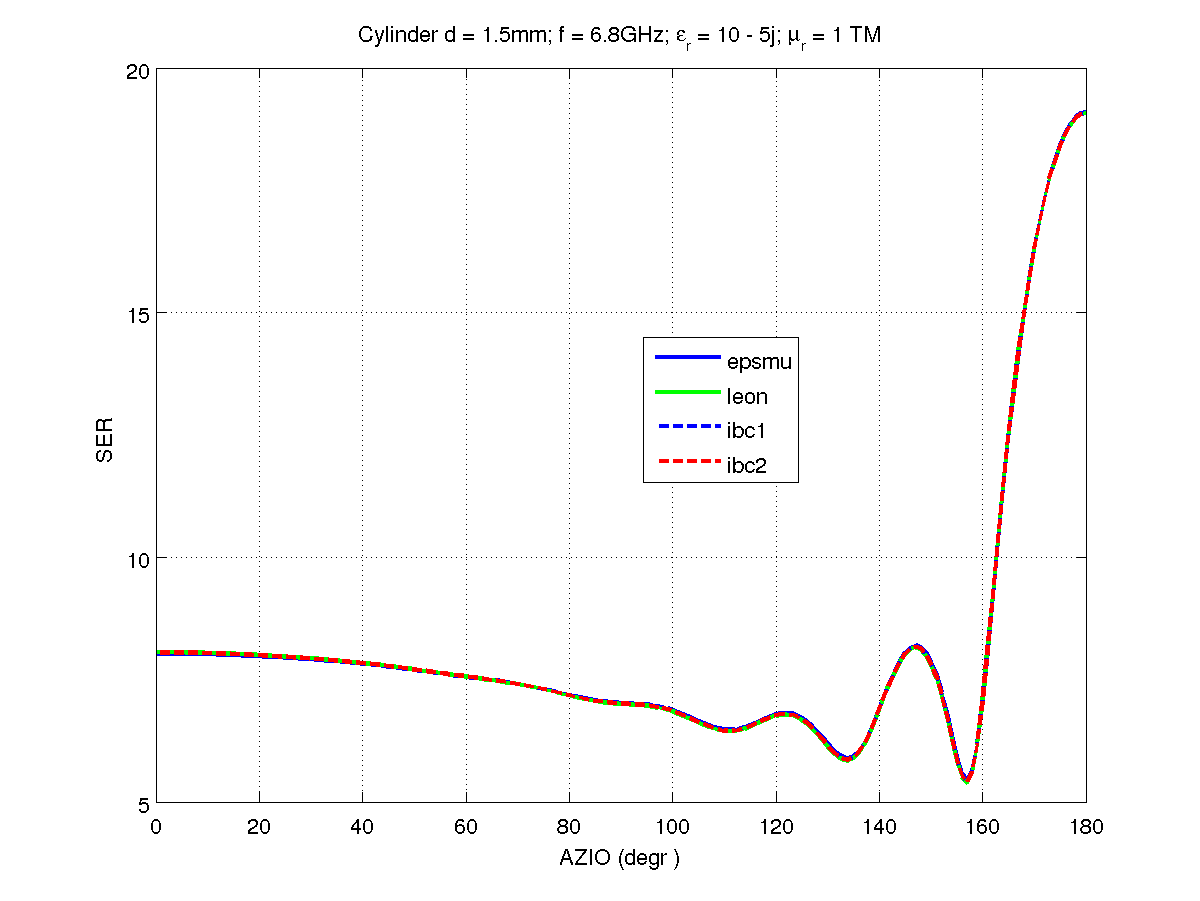}
 \subcaption{TM polarization}
\end{figure}
\end{minipage}
 \caption{Bistatic RCS for a coated circular cylinder with $d = 1.5mm$, $\epsilon_r = 10 - 5i$, $\mu_r = 1$ and $f = 6.8GHz$}
  \label{cyl_ep1p5_TM_TE}
\end{figure}

After we increase thickness of a boundary and decrease frequency, so we considered bistatic RCS for different scattering angles. \Cref{cyl_ep3_TE_TM} shows the bistatic radar cross section for a coated conducting cylinder with inner radius $\lambda_0$, coating thickness $d = 3mm$ and frequency $f = 3.4GHz$, coating parameters $\epsilon_r = 10 - 5i$ and $\mu_r = 1.0$, in TE and TM polarizations. The exact series solution is presented along with the SIBC and HOIBC order 1 and order 2 solutions.

\begin{figure}[H]
\begin{minipage}{0.45\textwidth}
\begin{figure}[H]
\centering
\includegraphics[scale=0.4]{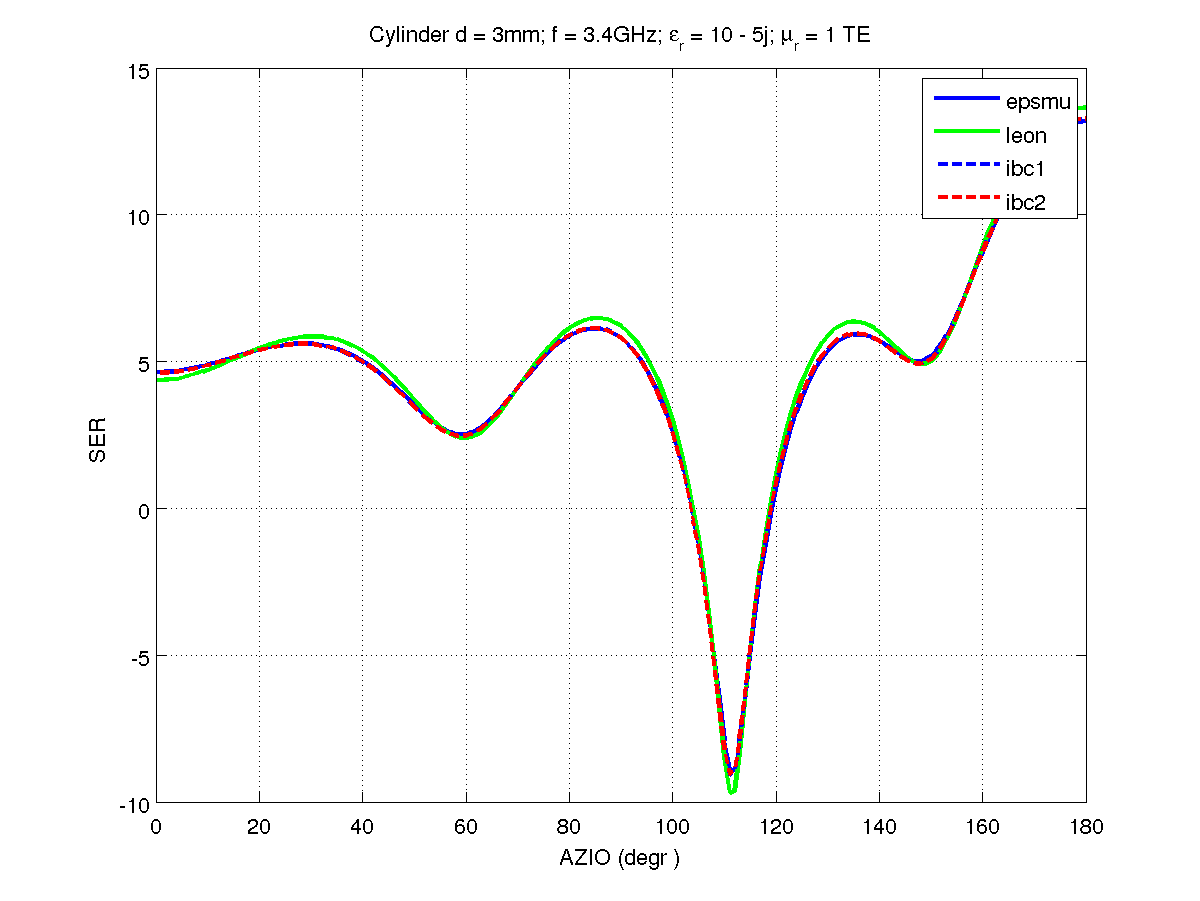}
 \subcaption{TE polarization}
\end{figure}
\end{minipage}   
\begin{minipage}{0.45\textwidth}
\begin{figure}[H]
\centering
\includegraphics[scale=0.38]{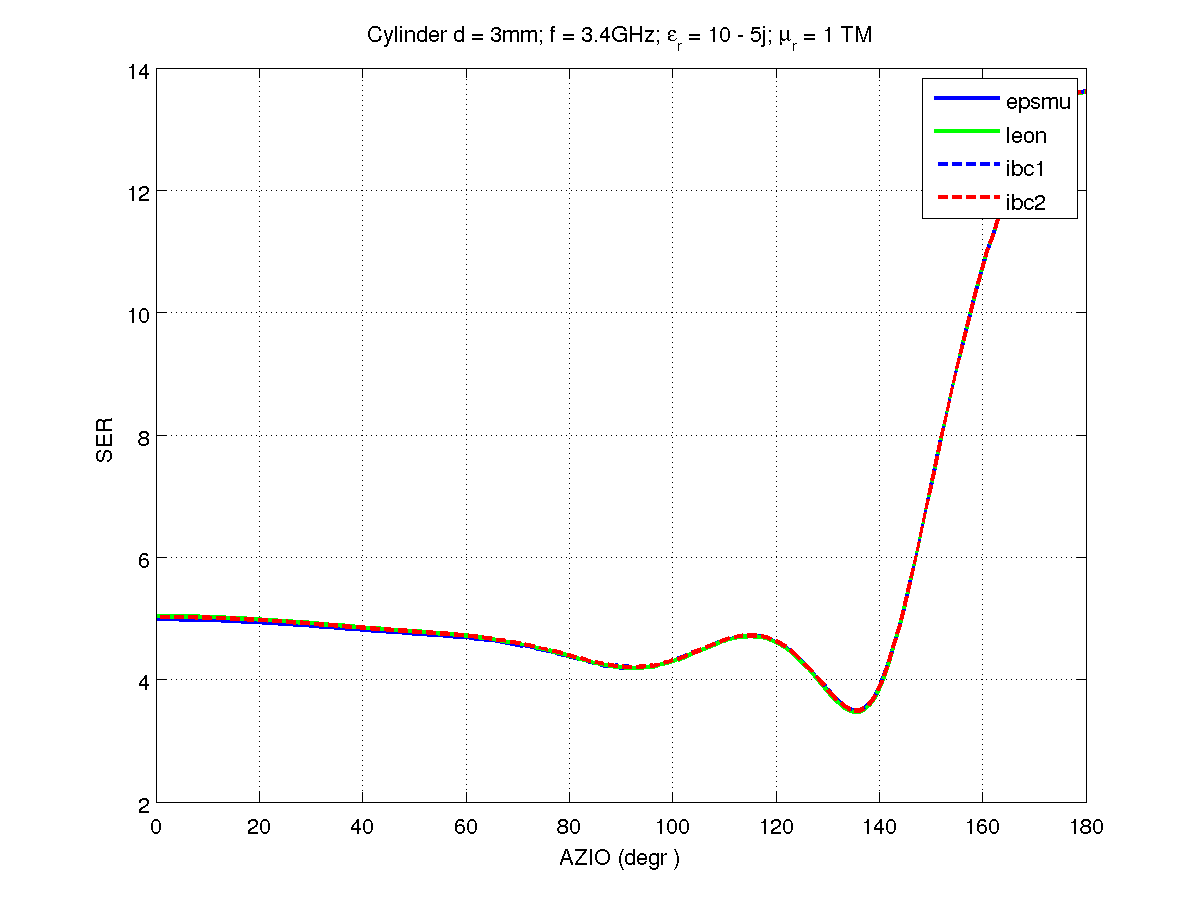}
 \subcaption{TM polarization}
\end{figure}
\end{minipage}
\caption{Bistatic RCS for a coated circular cylinder with TM polarization}
\label{cyl_ep3_TE_TM}
\end{figure}
Here we comput bistatic RCS for coated circular cylinder with parameters, $d = 0.1\lambda_0, \epsilon_r = 4-0.5i$ and $\mu_r = 1$. And we compare to Rahmat-Samii results for same test. The backscatter direction is $\phi = 180^{\circ} $. Results for exact formulation, IBC0 or Leontovich IBC formulation and the formulation based on the planar higher order IBC are presented in the figure {\ref{comp_cyl_hoppe}}. As can be seen in the figure, the results using the planar HOIBC are in excellent agreement with the exact solution over most of the angular range, while IBC0 solutions give only the average behavior of the scattered field.

\begin{figure}[H]
  \begin{center}
    \includegraphics[width=0.7\textwidth]{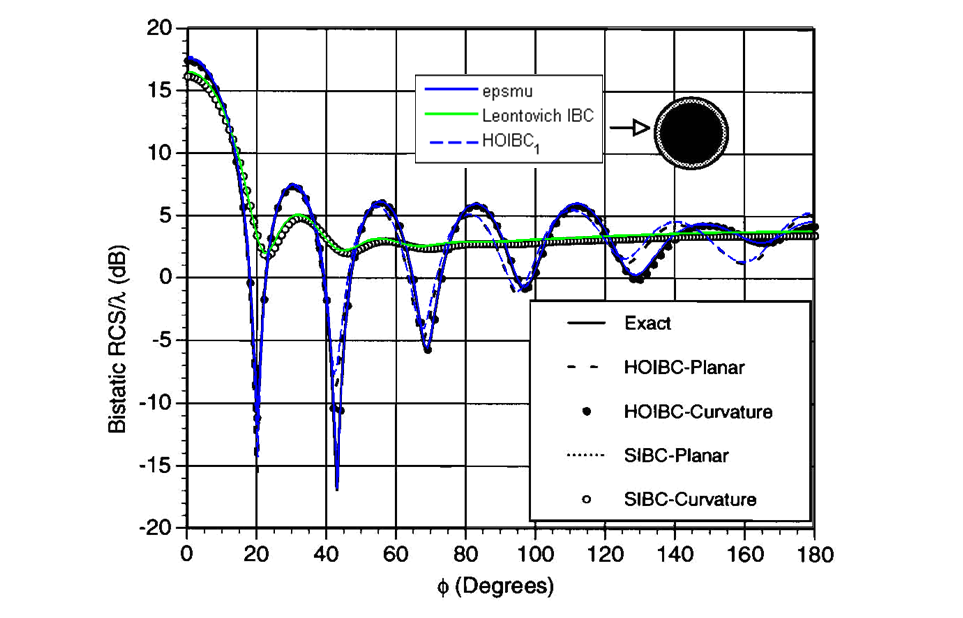}
  \end{center}
  \caption{Bistatic RCS for a coated circular cylinder, when $d = 0.1\lambda_0$, $\epsilon_r = 4-0.5j$, $\mu_r = 1$ with TE polarization}
  \label{comp_cyl_hoppe}
\end{figure}

Next we consider conducting plate with open boundary thin dielectric layer (see fig. \ref{fig:2Dplate}). Figures \ref{2Dpl_6p8_te}-\ref{2Dpl_6p8_tm} show the bistatic RCS for layer thickness $d = 4mm$ and frequency $f = 6.8GHz$. This example is interesting because it shows that method works even for open boundaries. And we can see that it solves problem much better than with Leontovich IBC. But it is difficult to see difference between first order and second order IBCs.
\begin{figure}[H]
\begin{minipage}{0.45\textwidth}
\begin{figure}[H]
\centering
\includegraphics[scale=0.4]{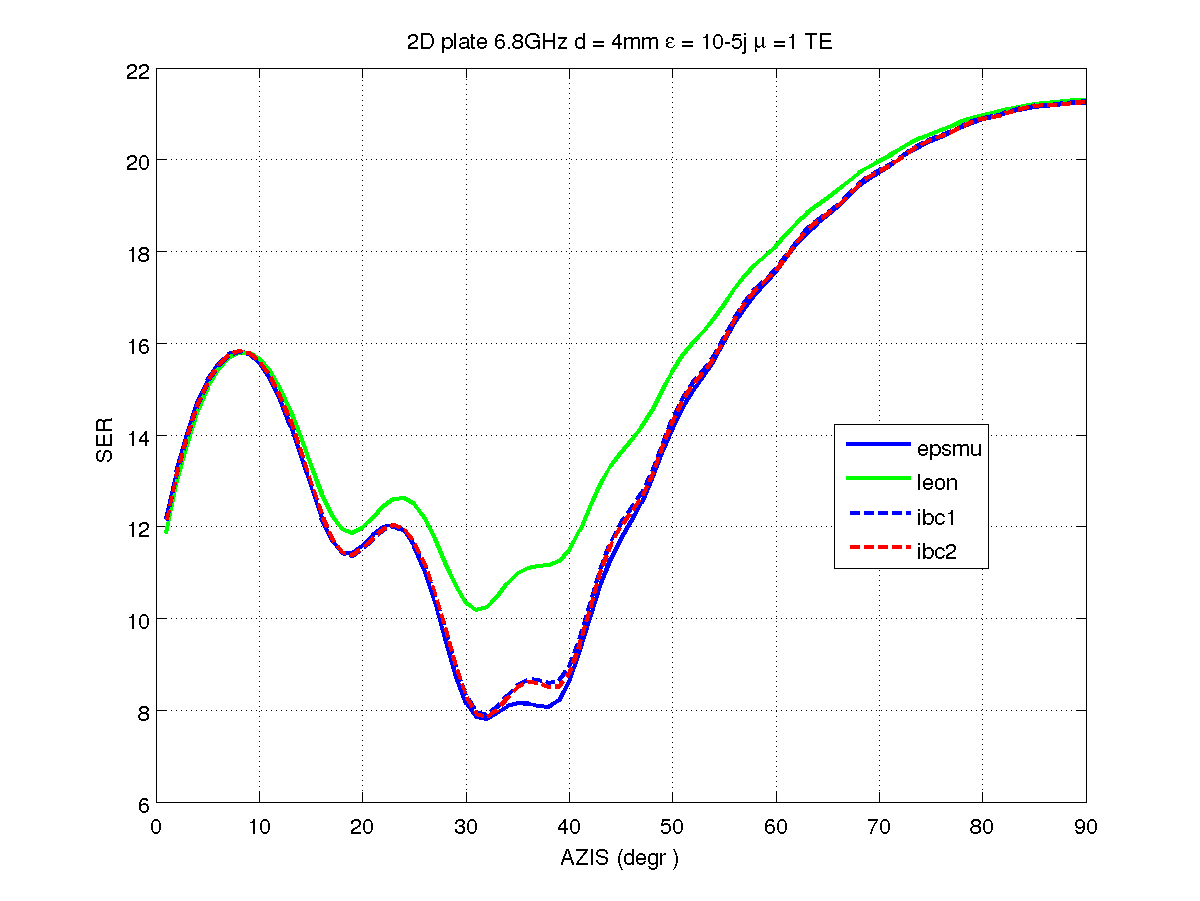}
 \subcaption{TE polarization}
  \label{2Dpl_6p8_te}
\end{figure}
\end{minipage}   
\begin{minipage}{0.45\textwidth}
\begin{figure}[H]
\centering
\includegraphics[scale=0.38]{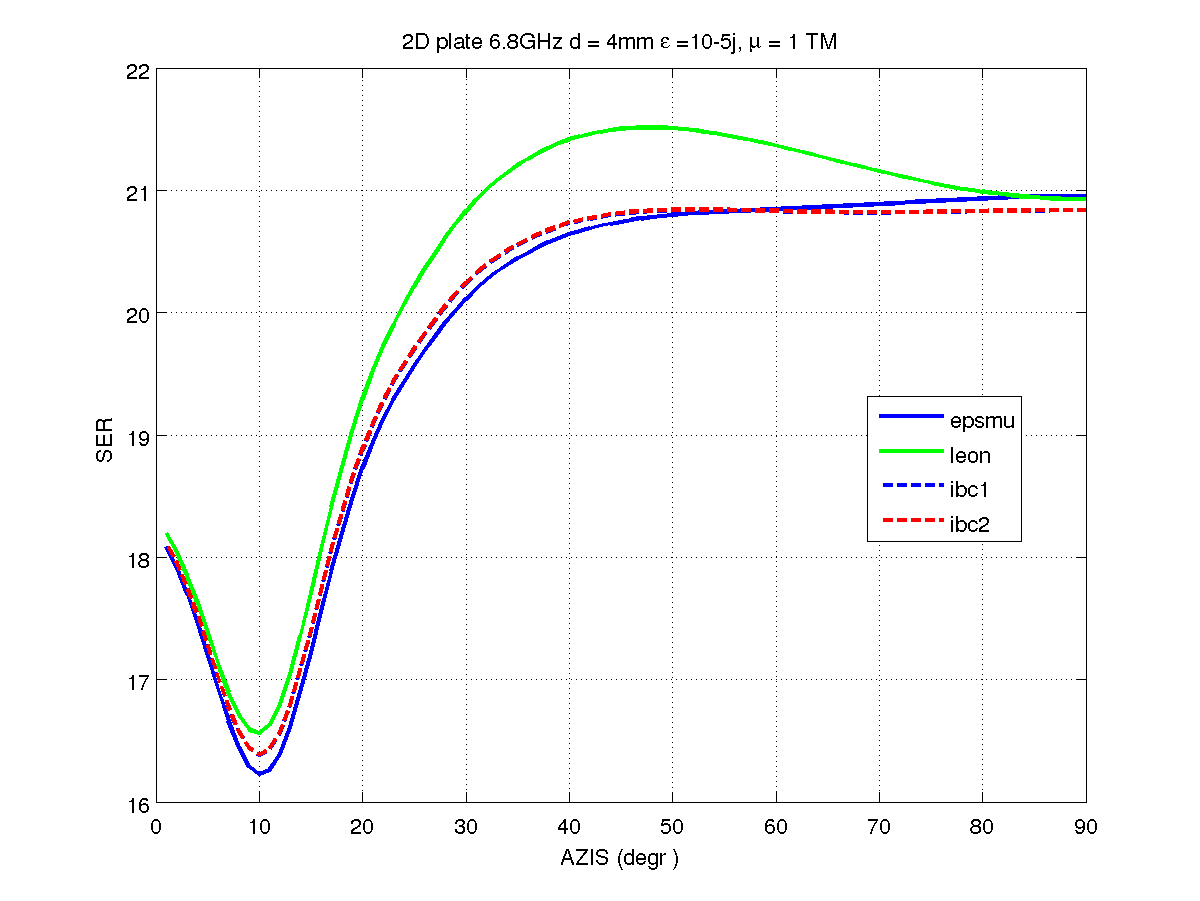}
\subcaption{TM polarization}
  \label{2Dpl_6p8_tm}
\end{figure}
\end{minipage}
 \caption{Bistatic RCS for a coated 2D plate }
\end{figure}
Then we propose a new variational formulation where electric and magnetic currents are the principal unknowns and in which we sum the weak form of the HOIBC with the Electric Field Integral Equation (EFIE) and use the Magnetic Field Integral Equation (MFIE) as a second equation.

 \section{Variational formulation for 3D problem }\label{variational_formulation}



Moreover, we have the boundary condition on $\Gamma$:
\begin{equation}\label{Z_3D_boxed}
(I + b_1 L_D - b_2 L_R) (\mathbf{n} \times \textbf{M}) = (a_0 I + a_1 L_D - a_2 L_R)\mathbf{J}.
\end{equation}
We sum the weak form of the HOIBC with the EFIE and use the MFIE as in (\ref{EFIE}) and (\ref{MFIE}) and using properties of vector analysis \cite{Van_Bladel_2007},
we get $$A((\mathbf{J}, \mathbf{M}), (\boldsymbol{\Psi}_J, \boldsymbol{\Psi}_M))  = F((\boldsymbol{\Psi}_J, \boldsymbol{\Psi}_M)) $$
where
$$ F((\boldsymbol{\Psi}_J, \boldsymbol{\Psi}_M)) = \int_{\Gamma} \mathbf{E}^{inc} \cdot \boldsymbol{\Psi}_J ds + \int_{\Gamma} \mathbf{H}^{inc} \cdot \boldsymbol{\Psi}_M ds$$
and
 $$A((\mathbf{J}, \mathbf{M}), (\boldsymbol{\Psi}_J, \boldsymbol{\Psi}_M)) = <Z_0(B-S) \textbf{J}, \boldsymbol{\Psi}_J > + \frac{1}{Z_0} < (B-S) \textbf{M}, \boldsymbol{\Psi}_M > $$
 $$ + <Q\textbf{M}, \boldsymbol{\Psi}_J> - <Q\textbf{J}, \boldsymbol{\Psi}_M > + \frac{a_0}{2} < \textbf{J}, \boldsymbol{\Psi}_J>
 + \frac{1}{2 a_0} < \textbf{M}, \boldsymbol{\Psi}_M > $$
 $$ - \frac{a_1}{2} < \div_{\Gamma} \textbf{J}, \div_{\Gamma} \boldsymbol{\Psi}_J >
 - \frac{a_2}{2} < \div_{\Gamma}(\mathbf{n} \times \textbf{J} ), \div_{\Gamma}(\mathbf{n} \times \boldsymbol{\Psi}_J) >$$
 $$ + \frac{b_1}{2} < \div_{\Gamma}(\mathbf{n} \times \textbf{M}), \div_{\Gamma} \boldsymbol{\Psi}_J>
 - \frac{b_2}{2} < \div_{\Gamma}\textbf{M}, \div_{\Gamma}(\mathbf{n} \times \boldsymbol{\Psi}_J) > $$
 $$- \frac{b_1}{2 a_0} < \div_{\Gamma}(\mathbf{n} \times \textbf{M}), \div_{\Gamma}(\mathbf{n} \times \boldsymbol{\Psi}_M ) >
 - \frac{b_2}{2 a_0} <\div_{\Gamma} \textbf{M}, \div_{\Gamma} \boldsymbol{\Psi}_M >$$
 $$+ \frac{a_1}{2 a_0} < \div_{\Gamma} \textbf{J}, \div_{\Gamma}(\mathbf{n} \times \boldsymbol{\Psi}_M) >
 - \frac{a_2}{2 a_0} < \div_{\Gamma}(\mathbf{n} \times \textbf{J}), \div_{\Gamma} \boldsymbol{\Psi}_M > $$
 %

 As $\mathbf{A} \in H^{-1/2}(\div, \Gamma)$, the  $\rot_{\Gamma}\mathbf{A}$ and $\div_{\Gamma}(\mathbf{n}\times\mathbf{A})$ is problematic. To overcome
this difficulty
 we introduce Lagrange multipliers $\boldsymbol{\lambda}_J $ and $\boldsymbol{\lambda}_M$ as in \cite{Bendali_OCT_1999}:
 %
 $$\int_{\Gamma} \boldsymbol{\lambda}_J \cdot (\tilde{\boldsymbol{\Psi}}_J - \mathbf{n}\times\boldsymbol{\Psi}_J) ds = 0 \;\; \forall (\tilde{\boldsymbol{\Psi}}_J,(\boldsymbol{\Psi}_J), $$
 $$\int_{\Gamma} \boldsymbol{\lambda}_M \cdot (\tilde{\boldsymbol{\Psi}}_M - \mathbf{n}\times\boldsymbol{\Psi}_M) ds = 0, \;\; \forall (\tilde{\boldsymbol{\Psi}}_M,(\boldsymbol{\Psi}_M),$$
 and we have:
 $$\int_{\Gamma} \boldsymbol{\lambda}'_J \cdot (\tilde{\mathbf{J} } - \mathbf{n}\times\mathbf{J}) ds = 0 \;\; \boldsymbol{\lambda}'_J  $$
 $$\int_{\Gamma} \boldsymbol{\lambda}'_M \cdot (\tilde{\mathbf{M}} - \mathbf{n}\times \mathbf{M}) ds = 0 \;\; \forall \boldsymbol{\lambda}'_M  $$
 %
 Finally, we have the following problem:
\begin{problem}\label{prob_3D}
 Find $U = (\mathbf{J}, \mathbf{M}, \tilde{\mathbf{J}}, \tilde{\mathbf{M}}) \in V = [H^{-1/2}(div, \Gamma) \cap L^2(\Gamma)]^4$ and $\lambda
 = (\boldsymbol{\lambda}_J, \boldsymbol{\lambda}_M) \in [H^{-1/2} (\Gamma)]^2$ such that 

 \begin{equation}\label{eq_prob_3D}
 \begin{cases}
 A(U, \Psi) + B^T(\lambda, \Psi) = F(\Psi) \\
   B(U, \lambda') = 0
 \end{cases}
 \end{equation}
 for all $\Psi = (\boldsymbol{\Psi}_J, \boldsymbol{\Psi}_M, \tilde{\boldsymbol{\Psi}}_J, \tilde{\boldsymbol{\Psi}}_M) \in V = [H^{-1/2}(div,
 \Gamma) \cap L^2(\Gamma)]^4$ and $\lambda' = (\boldsymbol{\lambda}'_J, \boldsymbol{\lambda}'_M) \in W = [H^{-1/2} (\Gamma)]^2$.
 \end{problem}
\noindent The bilinear forms are defined by:
 $$ B(U, \lambda') = \int_{\Gamma} \boldsymbol{\lambda}'_J \cdot (\tilde{\mathbf{J} } - \mathbf{n}\times\mathbf{J}) ds
 + \int_{\Gamma} \boldsymbol{\lambda}'_M \cdot (\tilde{\mathbf{M}} - \mathbf{n}\times \mathbf{M}) ds $$
 and
 \begin{equation}\label{oper_A_3d}
  A(U, \Psi) = iZ_0\iint_{\Gamma} kG\ (\textbf{J} \cdot \boldsymbol{\Psi}_J) - \frac{1}{k} G\ \div{\boldsymbol{\Psi}_J}\
  \div{\textbf{J}} ds ds'
 \end{equation}
 %
 %
 $$+ \frac{i}{Z_0} \iint_{\Gamma} kG\ (\boldsymbol{\Psi}_M \cdot \textbf{M}) - \frac{1}{k} G\ \div{\boldsymbol{\Psi}_M}\
 \div{\textbf{M}} dsds' $$
 $$+\iint_{\Gamma} \nabla'G \cdot (\boldsymbol{\Psi}_J \times \textbf{M}) dsds' - i \iint_{\Gamma} \nabla'G
 \cdot (\boldsymbol{\Psi}_M \times \textbf{J}) dsds'$$
 $$+ \frac{a_0}{2} \int_{\Gamma} \textbf{J} \cdot \boldsymbol{\Psi}_J ds + \frac{1}{2 a_0} \int_{\Gamma} \textbf{M}
 \cdot \boldsymbol{\Psi}_M ds $$
 $$ - \frac{a_1}{2} \int_{\Gamma} \div_{\Gamma}\textbf{J}\ \div_{\Gamma}\boldsymbol{\Psi}_J ds  - \frac{a_2}{2}
 \int_{\Gamma} \div_{\Gamma}\tilde{\textbf{J}}\ \div_{\Gamma} \tilde{\boldsymbol{\Psi}}_J ds$$
 $$ + \frac{b_1}{2} \int_{\Gamma} \div_{\Gamma}\tilde{\textbf{M}}\ \div_{\Gamma}\boldsymbol{\Psi}_J ds
 - \frac{b_2}{2} \int_{\Gamma} \div_{\Gamma}\textbf{M}\ \div_{\Gamma}\tilde{\boldsymbol{\Psi}}_J ds $$
 $$- \frac{b_1}{2 a_0} \int_{\Gamma} \div_{\Gamma}\tilde{\textbf{M}}\ \div_{\Gamma}\tilde{\boldsymbol{\Psi}}_M ds
 - \frac{b_2}{2 a_0} \int_{\Gamma} \div_{\Gamma}\textbf{M}\ \div_{\Gamma}\boldsymbol{\Psi}_M ds$$
 $$+ \frac{a_1}{2 a_0} \int_{\Gamma} \div_{\Gamma}\textbf{J}\ \div_{\Gamma}\tilde{\boldsymbol{\Psi}}_M ds
 - \frac{a_2}{2 a_0} \int_{\Gamma} \div_{\Gamma}\tilde{\textbf{J}}\ \div_{\Gamma}\boldsymbol{\Psi}_M ds. $$

 In the next section, we study mathematically this problem.

 \subsection{Existence and uniqueness theorem}\label{Existence_uniqueness}
We use a theorem from \cite{JCN} to prove existence and uniqueness of a solution to the problem \ref{prob_3D}. In the fisrt time we prove the following result.
 \begin{lem}\label{lemma_cont_A}
 The operator $A$ is continuous on $V\times V$ for all $\Psi \in V$ and we have that
  $$ | A(U, \Psi) |  \leq C \|U\|_V \| \Psi \|_V$$
 \end{lem}
 \begin{proof}:
In order to prove this lemma, it is convenient to introduce a decomposition of $A$ as
 $A = A_1 + A_2 +A_3$\\
 where
 $$A_1(U, \Psi) =  <Z_0(B-S) \textbf{J}, \boldsymbol{\Psi}_J > + \frac{1}{Z_0} < (B-S) \textbf{M}, \boldsymbol{\Psi}_M > $$
 $$ + <Q\textbf{M}, \boldsymbol{\Psi}_J> - <Q\textbf{J}, \boldsymbol{\Psi}_M > + \frac{a_0}{2} < \textbf{J}, \boldsymbol{\Psi}_J>
 + \frac{1}{2 a_0} < \textbf{M}, \boldsymbol{\Psi}_M > $$
 $$A_2(U, \Psi) = - \frac{a_1}{2} < \div_{\Gamma} \textbf{J}, \div_{\Gamma} \boldsymbol{\Psi}_J >
 - \frac{a_2}{2} < \div_{\Gamma} \tilde{\textbf{J}}, \div_{\Gamma} \tilde{\boldsymbol{\Psi} }_J  >$$
 $$- \frac{b_1}{2 a_0} < \div_{\Gamma} \tilde{\textbf{M} }, \div_{\Gamma} \tilde{ \boldsymbol{\Psi} }_M >
 - \frac{b_2}{2 a_0} <\div_{\Gamma} \textbf{M}, \div_{\Gamma} \boldsymbol{\Psi}_M >$$
 and
 $$A_3(U, \Psi) = \frac{b_1}{2} < \div_{\Gamma} \tilde{\textbf{M}}, \div_{\Gamma} \boldsymbol{\Psi}_J>
 - \frac{b_2}{2} < \div_{\Gamma}\textbf{M}, \div_{\Gamma} \tilde{\boldsymbol{\Psi} }_J  > $$
 $$+ \frac{a_1}{2 a_0} < \div_{\Gamma} \textbf{J}, \div_{\Gamma} \tilde{\boldsymbol{\Psi} }_M >
 - \frac{a_2}{2 a_0} < \div_{\Gamma} \tilde{ \textbf{J} }, \div_{\Gamma} \boldsymbol{\Psi}_M > $$
Hence, we get by theorems $2.2$ and $4.6$ in \cite{TL}  that there exists constants $C_1$ and $C_2$ such that
\begin{equation}\label{A_1}	
| A_1(U, \Psi) | \leq C_1 \| U \|_V \| \Psi \|_V,
\end{equation}
and
\begin{equation}\label{A_2}
|A_2(U, \Psi) + A_3(U, \Psi)|\leq C_2 \| U \|_V \|\Psi \|_V
\end{equation}
Combining (\ref{A_1}) and (\ref{A_2}) we obtain:
  $$ | A(U, \Psi) | = |A_1 + A_2 + A_3| \leq |A_1(U, \Psi)| + |A_2(U, \Psi) + A_3(U, \Psi)| \leq C \|U\|_V \| \Psi \|_V$$
where $C = C_1 + C_2$. 
\end{proof}
  We have also the lemma.
\begin{lem}\label{lemma_coer_A}
The operator $A$ is coercive on $V$.
We have to show that there exist $\alpha > 0$ such that
 $$ \Re[A(U, U^*)] \geq \alpha \|U\|^2_V - C\|U\|^2_V, \ \forall U\in V.$$
 \end{lem}
\begin{proof}:
From \cite{TL}, it follow that there exists $\alpha$ such that
 $$\Re(A_1) = \Re(<Z_0(B-S)\mathbf{J}, \mathbf{J}^*>) + \Re( <Z^{-1}_0(B-S)\mathbf{M}, \mathbf{M}^* >) + \Re(<Q \mathbf{M}, \mathbf{J}^*>)$$
 $$ - \Re(<Q \mathbf{J}, \mathbf{M}^* >) + \Re(\frac{a_0}{2}\int_{\Gamma} \mathbf{J} \cdot \mathbf{J}^* ds) +
 \Re(\frac{1}{2 a_0}\int_{\Gamma} \mathbf{M} \cdot \mathbf{M}^* ds) $$
 $$ \geq \alpha \left( \|\mathbf{J}\|^2_{-1/2, \div_{\Gamma}} + \| \mathbf{M} \|^2_{-1/2, \div_{\Gamma}} \right) +
\frac{\Re(a_0)}{2} \|\mathbf{J}\|^2_{L^2(\Gamma)} + \frac{\Re(a_0)}{2 |a_0|^2} \|\mathbf{M}\|^2_{L^2(\Gamma)}. $$
We can easily show that
$$\Re(A_2)=- \frac{\Re(a_1)}{2} \| \div_{\Gamma} \mathbf{J} \|^2_{L^2(\Gamma)}
- \frac{\Re(a_2)}{2} \| \div_{\Gamma} \tilde{\mathbf{J}} \|^2_{L^2(\Gamma)} $$
$$- \frac{\Re(b_1 a^*_0)}{2 |a_0|^2} \| \div_{\Gamma} \tilde{\mathbf{M}} \|^2_{L^2(\Gamma)}
- \frac{\Re(b_2 a^*_0)}{2 |a_0|^2} \| \div_{\Gamma} \mathbf{M} \|^2_{L^2(\Gamma)} $$
For the rest, we do as follows:
 $$\Re(A_3) = \Re(\frac{b_1}{2}\int_{\Gamma} \div_{\Gamma}\tilde{\mathbf{M}}\ \div_{\Gamma} \mathbf{J}^* ds )
 - \Re(\frac{b_2}{2}\int_{\Gamma} \div_{\Gamma}\mathbf{M}\ \div_{\Gamma} \tilde{\mathbf{J}}^* ds )$$
 $$+ \Re(\frac{a_1}{2 a_0}\int_{\Gamma} \div_{\Gamma}\mathbf{J}\  \div_{\Gamma} \tilde{\mathbf{M}}^* ds )
 - \Re(\frac{a_2}{2 a_0}\int_{\Gamma} \div_{\Gamma}\tilde{\mathbf{J}}\ \div_{\Gamma}\mathbf{M}^* ds ) $$
 $$= \Re\left\{ \left( \frac{b_1}{2} + \frac{a^*_1}{2 a^*_0} \right) \int_{\Gamma} \div_{\Gamma}\tilde{\mathbf{M}}\
 \div_{\Gamma}\mathbf{J}^* ds \right\}- \Re\left\{ \left( \frac{b_2}{2} + \frac{a^*_2}{2 a^*_0} \right) \int_{\Gamma} \div_{\Gamma}\tilde{\mathbf{J}}^*\
 \div_{\Gamma}\mathbf{M} ds \right\}$$
 $$ = \Re \left\{ \int_{\Gamma} \frac{1}{|a_0|^{1/2}} \left( \frac{b_1}{2} + \frac{a^*_1 a_0}{2 |a_0|^2} \right)^{1/2}
 \div_{\Gamma} \tilde{\mathbf{M}}\ \cdot |a_0|^{1/2} \left( \frac{b_1}{2} + \frac{a^*_1 a_0}{2 |a_0|^2} \right)^{1/2} \div_{\Gamma} \mathbf{J}^* ds \right\} $$
 $$ - \Re \left\{ \int_{\Gamma} |a_0|^{1/2} \left( \frac{b_2}{2} + \frac{a^*_2 a_0}{2 |a_0|^2} \right)^{1/2} \div_{\Gamma}
 \tilde{\mathbf{J}}^* \cdot \frac{1}{|a_0|^{1/2}} \left( \frac{b_2}{2} + \frac{a^*_2 a_0}{2 |a_0|^2} \right)^{1/2}
 \div_{\Gamma} \mathbf{M} ds \right\}. $$
We also define $q_1$ and $q_2$ by
$$q_1 = b_1 |a_0| + a^*_1 a_0 /|a_0| \,\,\,\,\,\,\,\,\,q_2 = b_2 |a_0| + a^*_2 a_0 /|a_0|,$$
we find that
 $$\Re(A_3) \geq -\frac{|q_1|}{4} \| \div_{\Gamma} \mathbf{J} \|^2_{L^2(\Gamma)} -\frac{|q_1|}{4|a_0|^2} \| \div_{\Gamma}
 \tilde{\mathbf{M}} \|^2_{L^2(\Gamma)} $$
 $$ -\frac{|q_2|}{4} \| \div_{\Gamma} \tilde{\mathbf{J}} \|^2_{L^2(\Gamma)} -\frac{|q_2|}{4|a_0|^2} \|
 \div_{\Gamma} \mathbf{M} \|^2_{L^2(\Gamma)}. $$
Using the conditions on coefficients
$$\Re(a_j) + \frac{|q_j|}{2} = 0,\,\,\,\,\,\,\, \text{where} \,\,\,j = 1,2$$
and from the sufficient uniqueness conditions we have that $\Re(a_j) = \Re(b_j^* a_0)$. \\
Thus, we obtain
$$\Re(A_2) + \Re(A_3) \geq 0 .$$
Finally we have for operator $A$ that
$$ \Re(A)\geq \alpha \left( \|\mathbf{J}\|^2_{-1/2, \div_{\Gamma}} + \| \mathbf{M} \|^2_{-1/2, \div_{\Gamma}} \right) + \frac{\Re(a_0)}{2} \|\mathbf{J}\|^2_{L^2(\Gamma)} + \frac{\Re(a_0)}{2 |a_0|^2} \|\mathbf{M}\|^2_{L^2(\Gamma)} $$
 \end{proof}
 \begin{lem}\label{lemma_LBB}
The operator $B$ verifies the inequality
$$ \sup_{\|U\|_V=1}|B(U, \lambda)| \geq \beta \| \lambda \|_W, \ \forall \lambda \in W = [H^{-1/2} (\Gamma)]^3 \times [H^{-1/2} (\Gamma)]^3 $$
where $U \in V = [H^{-1/2}(div, \Gamma) \cap L^2(\Gamma)]^4$ and $\beta > 0$.
\end{lem}
\begin{proof}: Here we have to show that there exists $\beta > 0$ such that
 $$ \sup_{\|U\|_V=1} \left| \int_{\Gamma} \boldsymbol{\lambda}_J \cdot (\tilde{\mathbf{J}} - n\times \mathbf{J}) + \boldsymbol{\lambda}_M \cdot (\tilde{\mathbf{M}} - n\times \mathbf{M}) ds \right| \geq \beta \| \lambda \|_W $$
 First, we take
 $$ \mathbf{J} = 0;\ \mathbf{M} = 0;\ \tilde{\mathbf{J}} = \frac{\hat{\mathbf{J}}}{\| \hat{\mathbf{J}} \|_V}\ \ and\ \ \tilde{\mathbf{J}} = \frac{\hat{\mathbf{M}}}{\| \hat{\mathbf{M}} \|_V} $$
$$\hat{\mathbf{J}} = \int_{\Gamma \setminus x} \frac{ \boldsymbol{\lambda}_J }{|x - y|} ds_y \ \ \ and\ \ \ \hat{\mathbf{M}} = \int_{\Gamma \setminus x} \frac{ \boldsymbol{\lambda}_M }{|x - y|} ds_y $$
 so, we get following inequality 
 $$\sup_{\|U\|_V=1} \left| \int_{\Gamma} \boldsymbol{\lambda}_J \cdot (\tilde{\mathbf{J}} - n\times \mathbf{J}) + \boldsymbol{\lambda}_M \cdot (\tilde{\mathbf{M}} - n\times \mathbf{M}) ds \right| \geq $$
 \begin{equation}\label{lambda_lambda}
 \frac{1}{\| \hat{\mathbf{J}} \|_{-1/2,\div_{\Gamma}} } \iint_{\Gamma \Gamma} \frac{\boldsymbol{\lambda}_J(x) \boldsymbol{\lambda}_J(y) }{|x - y|} ds_y ds_x + \frac{1}{\| \hat{\mathbf{M}}\|_{-1/2,\div_{\Gamma}} } \iint_{\Gamma \Gamma} \frac{\boldsymbol{\lambda}_M(x) \boldsymbol{\lambda}_M(y) }{|x - y|} ds_y ds_x 
 \end{equation}
 Using the Planchard-N\'ed\'elec inequality \cite{CD}, we get
 \begin{equation}\label{double_integrale}
  \iint_{\Gamma \Gamma} \frac{\mathbf{\lambda}(x) \mathbf{\lambda}(y) }{|x - y|} ds_y ds_x \geq \beta \| \mathbf{\lambda} \|^2_{-1/2, \Gamma}
 \end{equation}

and therefore, by (\ref{lambda_lambda}) and (\ref{double_integrale}), it follows that
$$\sup_{\|U\|_V=1} \left| \int_{\Gamma} \boldsymbol{\lambda}_J \cdot (\tilde{\mathbf{J}} - n\times \mathbf{J}) + \boldsymbol{\lambda}_M \cdot (\tilde{\mathbf{M}} - n\times \mathbf{M}) ds \right| \geq$$
 \begin{equation}\label{lambda_lambda_2}
    \frac{1}{\| \hat{\mathbf{J}} \|_{-1/2,\div_{\Gamma}}} \beta_J \| \mathbf{\lambda}_J \|^2_{-1/2, \Gamma} + \frac{1}{\| \hat{\mathbf{M}} \|_{-1/2,\div_{\Gamma}}} \beta_M \| \mathbf{\lambda}_M \|^2_{-1/2, \Gamma}
 \end{equation}
 furthermore, there exist $C_J > 0$ and $C_M > 0$ such that
 \begin{equation}
  \| \hat{\mathbf{J}} \|_{-1/2,\div_{\Gamma}} \leq C_J \| \mathbf{\lambda}_J \|_{-1/2, \Gamma} \ \ and \ \ \| \hat{\mathbf{M}} \|_{-1/2,\div_{\Gamma}} \leq C_M \| \mathbf{\lambda}_M \|_{-1/2, \Gamma}
 \end{equation}
 Consequently, we obtain that
 \begin{eqnarray*}
 \sup_{\|U\|_V=1} \left| \int_{\Gamma} \boldsymbol{\lambda}_J \cdot (\tilde{\mathbf{J}} - n\times \mathbf{J}) + \boldsymbol{\lambda}_M \cdot
(\tilde{\mathbf{M}} - n\times \mathbf{M}) ds \right| &\geq& \frac{\beta_J}{C_J} \| \mathbf{\lambda}_J \|_{-1/2, \Gamma} + \frac{\beta_M}{C_M} \|
\mathbf{\lambda}_M \|_{-1/2, \Gamma}\\
 &\geq& \beta \| \lambda \|_W
 \end{eqnarray*}
 where $\beta = \min( \beta_J/C_J; \beta_M/C_M)$.
\end{proof}
 \begin{theorem}
 The problem (\ref{prob_3D}) admits a unique solution $U \in V = [H^{-1/2}(div, \Gamma) \cap L^2(\Gamma)]^4$ and $\lambda \in [H^{-1/2} (\Gamma)]^2$, if coefficients satisfy
 \begin{equation}\label{eq_ex&un_coef_3d}
  \Re(a_j) + \frac{|a_0||b_j + a^*_j/a^*_0|}{2} = 0 \ \ for\ \ j=1,2.
 \end{equation}
 \end{theorem}

 \begin{proof}
Using the above lemma, we can show as in \cite{JCN} that the variational problem \ref{prob_3D} has a unique solution.
 \end{proof}
 In the next, a discretization of the problem \ref{prob_3D} is done with Reo-Wilton functions. 
\subsection{Discretization of the variational problem with HOIBC}\label{Discretization}
The first step is to approach the surface of the obstacle by a surface $\Gamma_h$ composed of finite number of two dimensional elements.
These elements are triangular facets denoted by $T_ i$ for $i=1$ to $N_T$:
$$\Gamma_h = \bigcup_{i=1}^{N_T} T_i .$$
We denote by $N_e$ the total number of edges of the mesh component $\Gamma_h$. 
Let $\{ \mathbf{f}_i \}_{i=1,N_e}$ be a function of Rao-Wilton-Glisson functions,\cite{RWG, AWG}. We decompose the electric and magnetic currents with thes basis functions:
$$\mathbf{J}(y) = \sum^{N_e}_{l=1} J_l \mathbf{f}_l(y),\ \ \mathbf{M}(y) = \sum^{N_e}_{l=1} M_l \mathbf{f}_l(y) ,$$
as well as auxiliary unknowns
$$\tilde{\mathbf{J}}(x) = \sum^{N_e}_{l=1} \tilde{J}_l \mathbf{f}_l(x),\ \ \tilde{\mathbf{M}}(x) = \sum^{N_e}_{l=1} \tilde{M}_l \mathbf{f}_l(x) .$$
For the  Lagrange multipliers, we use  basis functions proposed in \cite{Bendali_OCT_1999}:
\begin{equation}\label{relation_3}
\lambda_J = \sum^{N_e}_{k=1} \lambda_{J k} \mathbf{g}_k,\ \ \lambda_M = \sum^{N_e}_{k=1} \lambda_{M k} \mathbf{g}_k ,
\end{equation}
where $\mathbf{g_k}$ is defined as follows: \\
\begin{equation}\label{function_g_k}
\mathbf{g}_n(x) =\left\{
\begin{array}{cc}
(1 - 2\omega^+_{i+2}(x) )(\nu_n \times \mathbf{n}^+),\ \ \ \ x \in T^+_n \\
(1 - 2\omega^-_{i+2}(x) )(\nu_n \times \mathbf{n}^-),\ \ \ \  x \in T^-_n
\end{array} \right.
\end{equation}
with $\nu_n$ is a direction vector of the edge $\mathbf{n}$ 
and $\{ \omega_i \}_{i = 1,3}$ are barycentric coordinates of $x$ relative to triangles $T^+_n$ or $T^+_n$ (see \cite{JJin}).

Then, the discretized problem of  (\ref{eq_prob_3D}) is:
\begin{equation}\label{eq_sys_3D_h}
\begin{cases}
A^h(U_h, \Psi_h) + B^h(\lambda_h, \Psi_h) = \sum_{i=1}^{N_e} < \mathbf{E}^{inc}, \mathbf{f}_i > + \sum_{i=1}^{N_e} < \mathbf{H}^{inc}, \mathbf{f}_i > \\
B^h(U_h, \lambda_h) = 0
\end{cases}
\end{equation}
where
$$ A^h(U_h, \Psi_h) = \sum^{N_e}_{i,j = 1} <Z_0 (B-S) \textbf{f}_j, \mathbf{f}_i > J_j + Z^{-1}_0 \sum^{N_e}_{i,j = 1} < (B-S) \textbf{f}_j, \mathbf{f}_i > M_j $$
$$ + \sum^{N_e}_{i,j = 1} <Q \textbf{f}_j, \mathbf{f}_i> M_j - \sum^{N_e}_{i,j = 1} <Q\textbf{f}_j, \mathbf{f}_i > J_j + \frac{a_0}{2} \sum^{N_e}_{i,j = 1} < \textbf{f}_j, \mathbf{f}_i> J_j + \frac{1}{2 a_0} \sum^{N_e}_{i,j = 1} < \textbf{f}_j, \mathbf{f}_i > M_j $$
$$ - \frac{a_1}{2} \sum^{N_e}_{i,j = 1} < \mathbf{div}_{\Gamma}\textbf{f}_j, \mathbf{div}_{\Gamma}\mathbf{f}_i > J_j  - \frac{a_2}{2} \sum^{N_e}_{i,j = 1} < \mathbf{div}_{\Gamma}\textbf{f}_j, \mathbf{div}_{\Gamma} \mathbf{f}_i > \tilde{J}_j $$
$$ + \frac{b_1}{2} \sum^{N_e}_{i,j = 1} < \mathbf{div}_{\Gamma}\textbf{f}_j, \mathbf{div}_{\Gamma}\mathbf{f}_i > \tilde{M}_j - \frac{b_2}{2} \sum^{N_e}_{i,j = 1} < \mathbf{div}_{\Gamma}\textbf{f}_j, \mathbf{div}_{\Gamma}\mathbf{f}_i > M_j $$
$$ - \frac{b_1}{2 a_0} \sum^{N_e}_{i,j = 1} < \mathbf{div}_{\Gamma}\textbf{f}_j, \mathbf{div}_{\Gamma}\mathbf{f}_i > \tilde{M}_j - \frac{b_2}{2 a_0} \sum^{N_e}_{i,j = 1} < \mathbf{div}_{\Gamma}\textbf{f}_j, \mathbf{div}_{\Gamma}\mathbf{f}_i > M_j $$
$$ + \frac{a_1}{2 a_0} \sum^{N_e}_{i,j = 1} < \mathbf{div}_{\Gamma}\textbf{f}_j, \mathbf{div}_{\Gamma}\mathbf{f}_i > J_j - \frac{a_2}{2 a_0} \sum^{N_e}_{i,j = 1} < \mathbf{div}_{\Gamma} \textbf{f}_j, \mathbf{div}_{\Gamma}\mathbf{f}_i > \tilde{J}_j ;$$
and
$$ B^h (U_h, \lambda'_h) = \sum^{N_e}_{i, k = 1} < \mathbf{g}_k, \mathbf{f}_i > \tilde{J}_i - \sum^{N_e}_{i,k = 1} < \mathbf{g}_k, \mathbf{n}\times\mathbf{f}_i > J_i$$
$$ + \sum^{N_e}_{i,k = 1} < \mathbf{g}_k, \mathbf{f}_i > \tilde{M}_i - \sum^{N_e}_{i,k = 1} < \mathbf{g}_k, \mathbf{n}\times\mathbf{f}_i > M_i $$
Obviously, this problem has a unique solution since we use conforming finite element method.\\

Now, we express the problem in matrix form:
$$(B-S)_{i,j} = i\int \int_{\Gamma_h} k G(s,s') \mathbf{f}_j(s') \cdot \mathbf{f}_i(s) - \frac{1}{k} G(s,s') (\mathbf{div}_{\Gamma}\mathbf{f}_i) (\mathbf{div}_{\Gamma}' \mathbf{f}_j) ds ds' $$
$$Q_{i,j} = -i\int \int_{\Gamma_h} [\mathbf{f}_i(s) \times \mathbf{f}_j(s')] \cdot \nabla'_{\Gamma} G(s,s') ds ds' $$
$$I_{i,j} = \int_{\Gamma_h} \mathbf{f}_i \cdot \mathbf{f}_j ds ; \ \ \ \ D_{i,j} = \int_{\Gamma_h} (\mathbf{div}_{\Gamma}\mathbf{f}_j) (\mathbf{div}_{\Gamma}\mathbf{f}_i) ds $$
\begin{equation}\label{CH_CK}
C_{Hi,j} = \int_{\Gamma_h} \mathbf{g}_i \cdot \mathbf{f}_j ds ; \ \ \ \ C_{Ki,j} = \int_{\Gamma_h} \mathbf{g}_i \cdot (\mathbf{n} \times \mathbf{f}_j) ds,
\end{equation}
where $[C_H]$ is a nonsingular diagonal matrix. Therefore it is invertible.\\
Then, we define $[A1]$ and $[A2]$ by
\begin{equation}\label{A1_A2}
[A1] = [(B-S)] + \frac{a_0}{2}[I] - \frac{a_1}{2}[D],\ \ [A2] =[(B-S)] + \frac{1}{2 a_0}[I] - \frac{b_2}{2 a_0}[D].
\end{equation}
Using (\ref{CH_CK}) and (\ref{A1_A2}), we now present (\ref{eq_sys_3D_h}) in the following matrix form
\begin{equation}\label{system_ch_ck}
\left (
\begin{array}{cccccc}
[A1] & [Q] & 0 & \frac{b_1}{2}[D] & [C_K]^T & 0 \\
{[Q]}^T & [A2] & -\frac{a_2}{2 a_0}[D] & 0 & 0 & [C_K]^T \\
0 & -\frac{b_2}{2}[D] & -\frac{a_2}{2}[D] & 0 & [C_H]^T & 0 \\
\frac{a_1}{2 a_0}[D] & 0 & 0 & -\frac{b_1}{2 a_0}[D] & 0 & [C_H]^T \\
{[C_K]} & 0 & [C_H] & 0 & 0 & 0 \\
0 & [C_K] & 0 & [C_H] & 0 & 0
\end{array}
\right)
\left (
\begin{array}{cccccc}
\overline{J} \\ \overline{M} \\ \overline{\tilde{J}} \\ \overline{\tilde{M}} \\ \overline{\lambda_J} \\ \overline{\lambda_M}
\end{array}
\right) =
\left (
\begin{array}{cccccc}
\overline{E} \\ \overline{H} \\ 0 \\ 0 \\ 0 \\ 0
\end{array}
\right)
\end{equation}
where right-side vectors $\overline{E}$, $\overline{H}$ are defined as follows:
$$E_i = \int_{\Gamma_h} \mathbf{E}^{inc}\cdot\mathbf{f}_i ds ; \ \ \ \ H_i = \int_{\Gamma_h} \mathbf{H}^{inc}\cdot\mathbf{f}_i ds .$$
In the second step, we focus ours attention on elimination auxiliary currents and the Lagrange multiplier.
 Using the definition of the basis function (\ref{function_g_k}), we obtain that
$$\int_{\Gamma} \mathbf{g}_n(s) \cdot \tilde{\mathbf{J}} ds = \frac{|T^+_n| + |T^-_n|}{3} \tilde{J}_n .$$
On the other side, we have
$$\int_{\Gamma} \mathbf{g}_n(s) \cdot (\mathbf{n}\times\mathbf{J}) ds = \int_{T^+_n\cup T^-_n} \mathbf{g}_n(s) \cdot (\mathbf{n}\times\mathbf{J})  ds,$$
that is calculated with help of Gaussian quadrature.
And we get next equation
\begin{equation}\label{eq_chjt+ckj=0}
\frac{|T^+_n| + |T^-_n|}{3} \tilde{J}_n = \int_{T^+_n\cup T^-_n} \mathbf{g}_n(s) \cdot (\mathbf{n}\times\mathbf{J}) ds
\end{equation}
Thus, we obtain
\begin{equation}\label{eq_jt=-chinvckj}
\tilde{J}_n = \frac{3}{|T^+_n| + |T^-_n|} \int_{T^+_n\cup T^-_n} \mathbf{g}_n(s) \cdot (\mathbf{n}\times\mathbf{J}) ds
\end{equation}
on each edge of the mesh.
Consequently, we conclude that the auxiliary currents can be rewritten as
%
\begin{equation}\label{star}
\begin{array}{cc}
\overline{\tilde{J}} = -[C_H]^{-1} [C_K] \overline{J}\,\,\,\, \text{and} \,\,\,\,\overline{\tilde{M}} = -[C_H]^{-1} [C_K] \overline{M}
\end{array}
\end{equation}
%
So (\ref{system_ch_ck}) together with (\ref{star}) yields the following  expression of Lagrange multipliers in terms of $\overline{J}$ and $\overline{M}$
\begin{equation}\label{star_2}
\overline{\lambda_J} = \frac{b_2}{2} [C_H]^{-T} [D] \overline{M} - \frac{a_2}{2} [C_H]^{-T} [D] [C_H]^{-1} [C_K] \overline{J}
\end{equation}
\begin{equation}\label{star_3}
\overline{\lambda_M}= - \frac{a_2}{2 a_0} [C_H]^{-T} [D] \overline{J} - \frac{b_1}{2 a_0} [C_H]^{-T} [D] [C_H]^{-1} [C_K] \overline{M}
\end{equation}
Now we define  matrixes:
 \begin{equation}\label{star_4}
[C_{KH}] = [C_H]^{-1} [C_K] \,\,\,\, \text{and} \,\,\,\,  [C_{KH}]^T = [C_K]^T [C_H]^{-T}
\end{equation}

%
%
%
%
%

Then, by (\ref{star_2}),(\ref{star_3}) and (\ref{star_4}),  we obtain the final system
$$
\left(
\begin{array}{cc}
[A1] - \frac{a_2}{2} [C_{KH}]^T [D] [C_{KH}] & [Q] - \frac{b_1}{2} [D] [C_{KH}] + \frac{b_2}{2} [C_{KH}]^T [D] \\
\\
{[Q]}^T + \frac{a_2}{2 a_0}[D] [C_{KH}] - \frac{a_1}{2 a_0}[C_{KH}]^T [D] & [A2] - \frac{b_1}{2 a_0}[C_{KH}]^T [D] [C_{KH}]
\end{array}
\right)
\left(
\begin{array}{cc}
\overline{J} \\
\\
\overline{M}
\end{array}
\right) =
\left (
\begin{array}{cc}
\overline{E} \\
\\
\overline{H}
\end{array}
\right) $$

%

\subsection{Numerical results}\label{results}


We present now some numerical results obtained. As a first example we consider the case of a coated conducting sphere with a conductor radius of $1.5\lambda_0$ and coating thickness of $0.0075\lambda_0$, with $\eps_r=5$ and $\mu_r=1.0$.
 \Cref{rot1} show the $\theta \theta$  components of the bistatic RCS for a plane wave incident from $\theta=0$. Three solutions are included: the exact serie solution (MIE) and the solutions of the methods of moment studied below with SIBC and the HOIBC. The Figure clearly shows the increased accuracy of the HOIBC solution relative to the SIBC solution. The SIBC gives only the average behavior of the scattered field while the HOIBC accurately predicts the sidelobe behavior.

    \begin{figure}[H]
    \centering
        \includegraphics[scale=0.63]{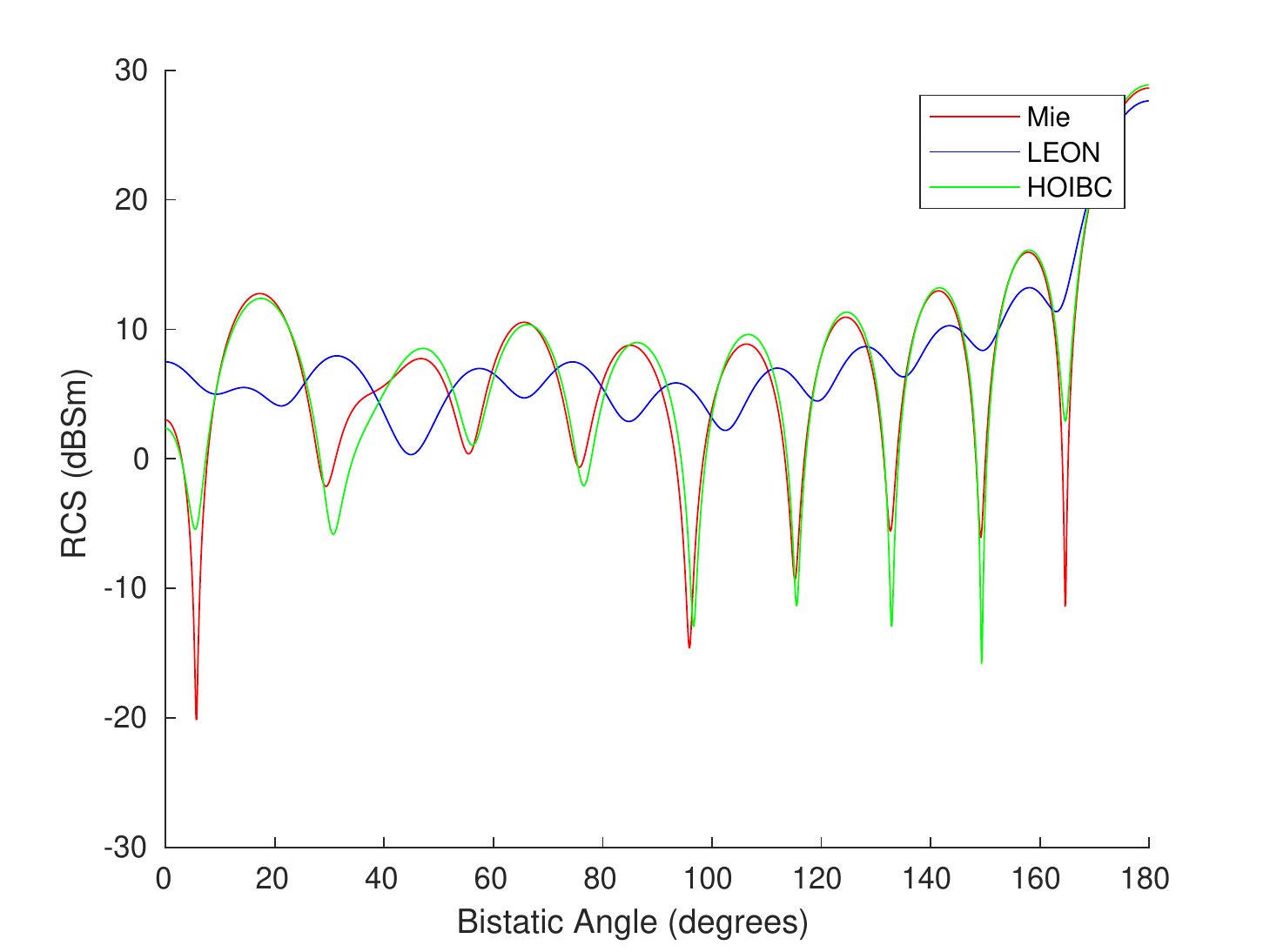}
        \caption{$\theta \theta$ component of the bistatic RCS for a coated conducting sphere with frequency $f= 0.45GHz$, layer thickness $d = 0.05m$, $\varepsilon_r = 5$ and $\mu_r = 1$.
        Exact serie solution and HOIBC solution.}
        \label{rot1}
    \end{figure}

The second test is a coated conducting spheroid whose the radii are 0.5m and 1m with a coating thickness of $0.17\lambda_0$, with $\eps_r=5$ and $\mu_r=1.0$.
 \Cref{rot3} and \Cref{rot4} show  the $\theta \theta$ and $\phi \phi$ components of the monostatic RCS. Three solutions are included : a method of moments solution called PMCHWT \cite{BEM}, the HOIBC solution and the SIBC solution. It does not exits exact solution for this case.
 We note that a slight difference between the PMCHWT solution and the HOIBC solution and the SIBC solution is very poor. 
 In both cases, excellent results are obtained for all angles of incidence. But, the finite radii of curvature on the spheroid contribute to the inaccuracy of the HOIBC result.

\begin{figure}[H]
\centering
        \includegraphics[scale=0.5]{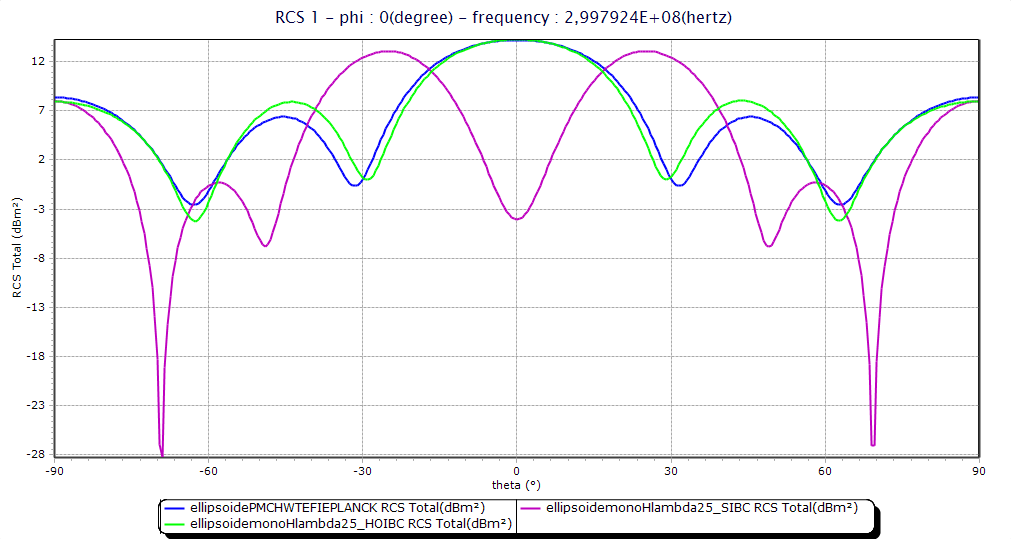}
        \caption{$\theta \theta$ component of the monostatic RCS for a coated conducting spheroid, PMCHWT solution, SIBC solution and HOIBC solution.}
        \label{rot3}
    \end{figure}

\begin{figure}[H]
\centering
        \includegraphics[scale=0.5]{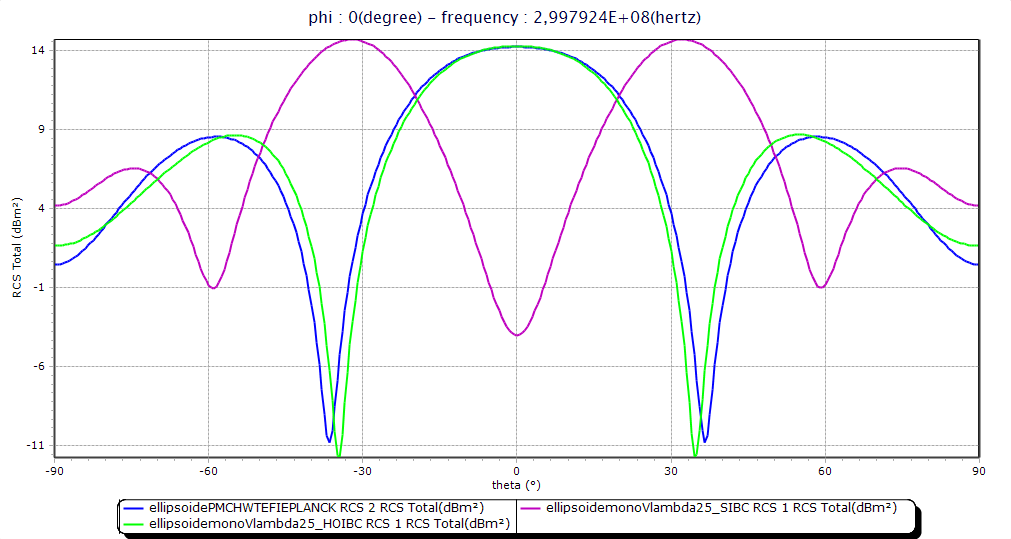}
        \caption{$\phi \phi$ component of the monostatic RCS for a coated conducting spheroid, PMCHWT solution, SIBC solution and HOIBC solution.}
        \label{rot4}
    \end{figure}

\section{Conclusion}\label{conclusion}

In this paper, we proposed a new formulation of electromagnetic scattering problem with higher order impedance boundary condition. We prove the existence and uniqueness of the solution
of the variational formulation. Then we discretize the formulation with RWG basis functions. The validations are shown an important accuracy improvement
of the HOIBC model over the SIBC model in the case of materials of lower index and/or with small losses.


\begin{thebibliography}{99.}

\bibitem{RFH} R. F.~Harrington,
\emph{Field Computation by Moment Methods}, {\sl  Wiley-IEEE Press},(1993).
\bibitem{ABENDL}, A.~Bendali and K.~Lemrabet,
\emph{The effect of a thin coating on the scattering of a time-harmonic wave for the Helmholtz equation},
{\sl SIAM J. Appl. Math.},pp.~1664-1693,(1996).
\bibitem{R-S_DEC} Y.~Rahmat-Samii and J. H.~Daniel,
\emph{Scattering by Superquadric Dielectric-Coated Cylinders Using Higher Order Impedance Boundary Conditions},
{\sl IEEE Trans. Antennas Propagat.},vol.~40,No.~12,pp.~1513-1522,Dec (1992).
\bibitem{OMBS} O.~Mareaux, B.~Stupfel,
\emph{High-order impedance boundary conditions for multilayer coated 3-D Ojects},
{\sl IEEE Transactions on Antennas and Propagations},Vol.~48,No.~3,pp.~429-436,March(2000).
\bibitem{Stupfel_NOV_2005} B.~Stupfel,
\emph{Impedance Boundary Conditions for Finite Planar or Curved Frequency Selective Surfaces Embedded in Dielectric Layers},
{\sl IEEE Trans. Antennas Propagat},Vol.~53,pp.~3654--3662,Nov (2005).
\bibitem{Stupfel_MARCH_2000}, B.~Stupfel and O.~Marceaux, \emph{High-Order Impedance Boundary Conditions for Multilayer Coated 3-D Objects}
{\sl IEEE TRANSACTIONS ON ANTENNAS AND PROPAGATION}, Vol.~48,N.~3 pp.~429-436, March(2000).
\bibitem{Stupfel_APR_2005} B.~Stupfel and Y.~Pion,
\emph{Impedance Boundary Conditions for Finite Planar or Curved Frequency Selective Surfaces},
{\sl IEEE Trans. Antennas Propagat},Vol.~53,pp.~1415--1424,April (2005).
\bibitem{Stupfel_JUNE_2011} B.~Stupfel, D.~Poget and J.~C.~Nédélec,
\emph{Sufficient uniqueness conditions for the solution of the time harmonic Maxwell's equations associated with surface impedance boundary conditions},
{\sl Journal of Computational Physics},Vol.~230, No~12,pp.~4571--4587, (2011).
\bibitem{Stupfel_2015} B.~Stupfel,
\emph{Implementation of High-Order Impedance Boundary Conditions in Some Integral Equation Formulations},
{\sl IEEE Trans. Antennas Propagat}, Vol.~63, NO.~4,April (2015).

\bibitem{Stupfel_2013} B.~Stupfel,
\emph{One-way domain decomposition method with adaptive absorbing boundary condition for the solution of Maxwell's equations},
{\sl IEEE Trans. Antennas Propagat}, Vol.~61, NO.~10,pp.~5100--5108, Aug (2015).

\bibitem{RYTOV_1940} S.~M.~Rytov,
\emph{Calcul du skin-effet par la méthode des perturbations},
{\sl J.~Phys.~USSR},Vol.~2,pp.~233--242, Oct (1940).
\bibitem{Van_Bladel_2007} J.~G.~Van Bladel,
\emph{Electromagnetic Fields},
{\sl Wiley}, Second, (2007).
\bibitem{Bendali_OCT_1999} A.~Bendali, M.~Fares and J.~Gay,
\emph{A Boundary-Element Solution of the Leontovich Problem},
{\sl IEEE Trans. Antennas Propagat.}, Vol.~47, pp.~1597--1605, October, (1999).
\bibitem{JCN} J. C.~Nédélec,
\emph{Acoustic and electromagnetic equations: Integral representations for harmonic problems},
{\sl Springer édition}, Vol.~144, (2000).
\bibitem{TL} V.~Lange,
\emph{Equations intégrales espace-temps pour les équations de Maxwell: calcul du champ diffracté par un obstacle dissipatif},
{\sl PhD thesis, Mathématiques appliquées, Bordeaux 1}, (1995).
\bibitem{CD} C.~Daveau and J.~Laminie,
\emph{Mixed and Hybrid Formulations For The Three-Dimensional Magnetic Problem},
{\sl Numerical Methods for Partial Differential Equations}, vol.~18, no.~1, pp.~85-104, January,(2002).
\bibitem{RWG} S.~M.~Rao,D.~R.~Wilton and  A.~W.~Glisson,
\emph{ Electromagnetic scattering by surfaces of arbitrary shapes},
{\sl IEEE Trans. Antennas Propag}, Vol.~30, No.~3, pp.~409--418, (1982).
\bibitem{AWG} A.~W.~Glisson,
\emph{ Electromagnetic scattering by arbitrary shapes surfaces with impedance boundary conditions},
{\sl Radio Sci}, Vol.~27, No.~6, pp.~935--943, May, (1992).
\bibitem{DSW} D.~S.~Wang,
\emph{ Limits and validity of the impedance boundary conditions on penetrable surfaces},
{\sl IEEE Trans. Antennas Propag}, Vol.~AP-35, No.~6, pp.~453--457, (1987).
\bibitem{Cicchetti_1996} R.~Cicchetti,
\emph{ A class of exact and higher-order surface boundary conditions for layered st},
{\sl IEEE Trans. Antennas Propag}, Vol.~AP-35, No.~6, pp.~453--457, (1987).
\bibitem{JJin} J.~Jin,
\emph{The Finite Element Method in Electromagnetics},
{\sl New York: Wiley}, (1993).
\bibitem{KARP_1967} S.~N.~Karp and F.~C.~Karal Jr,
\emph{ Generalized impedance boundary conditions with applications to surface wave structures},
{\sl Electromagnetic Wave Theory, J.~Brown, Ed. New York: Pergamon}, pp.~479--483, (1967).
\bibitem{senior_1995} T.~B.~A.~Senior and J.~L.~Volakis,
\emph{ Approximate boundary conditions in electromagnetics},
{\sl Inst. Elect. Eng. Electromagn. Waves Series 41},  (1995).
\bibitem{Abil} A.~Aubakirov,
\emph{Electromagnetic scattering problem with hight order impedance boundary condition and integral methods}, {\it PHD Thesis\/},
Applied Mathematics, University of Cergy-Pontoise, January (2014).
%
 \bibitem{BEM} Medgyiesi-Metschang, L. N., Putnam J.M. and  Gedera M. B.,
 \emph{Generalized method of moments for three-dimensional penetrable scatterers},
 {\sl J. Opt. Soc. Am. A \/}, Vol.~11,No.~4, (1994).

\bibitem{R-S} Y.~Rahmat-Samii  and J.H.~Daniel,
\emph{Impedance Boundary Conditions in Electromagnetics}, {\sl Taylor \& Francis}, (1995).
%
\bibitem{JR} J. R. Rogers,
\emph{Moment method scattering solutions to impedance boundary condition integral equations}, 
{\sl IEEE AP-S Int. Symp., Boston, MA}, pp.~347--350, (1984).
%
\bibitem{Senior-Volakis} T.B.A.~Senior and  J.L.~Volakis, %
\emph{Approximate boundary conditions in electromagnetics}, {\sl IEE Electromagnetic Waves Series 41}, (1995).
%
%
\bibitem{Lange-1995} V.~Lange,%
\emph{Equations int\'{e}grales espace-temps pour les \'{e}quations de Maxwell : calcul du champ diffract\'{e} par un obstacle dissipatif} , PHD thesis {\sl Math\'{e}matiques appliqu\'{e}es}, Bordeaux, 1, (1995).
%
%
\bibitem{Stratton-Chu} G. C.~Hsiao and R. F.~Kleinman,
\emph{Mathematical foundations for error estimations in numerical solutions of integral equations in electromagnetics} , 
{\sl IEEE Trans. Antennas Propagat.}, vol.45, pp.316-328, Mar.(1997).
\bibitem{TE} I.~Terrasse,
\emph{Résolution mathémathique et numérique des équations de Maxwell instationnaires par une méthode de potentiels retardés} , 
{\sl L'école Polytechnique}, .(1993).
\end{thebibliography}
\end{document}